\documentclass[11pt,a4paper,reqno]{amsart}

\usepackage[T1]{fontenc}
\usepackage[english]{babel}
\usepackage{mathtools}
\usepackage{amssymb}
\usepackage{libertinus}
\usepackage[cal=boondoxo,bb=ams]{mathalfa}
\usepackage{microtype}
\usepackage{enumitem}
\usepackage{etoolbox}
\usepackage{needspace}
\usepackage{xcolor}

\allowdisplaybreaks[2]
\usepackage[
  a4paper,
  left=30mm,
  right=30mm,
  top=27mm,
  bottom=30mm,
  headsep=8mm,
  footskip=13mm,
  heightrounded
]{geometry}

\setlist[itemize]{leftmargin=2.1em,itemsep=0.25em,topsep=0.45em,parsep=0pt}
\setlist[enumerate]{leftmargin=2.1em,itemsep=0.25em,topsep=0.45em,parsep=0pt}

\numberwithin{equation}{section}

\makeatletter
\def\@settitle{%
  \begin{center}
    \vspace*{-0.4em}
    {\normalfont\bfseries\fontsize{17}{21}\selectfont\@title\par}
    \vspace{0.7em}
  \end{center}}
\renewcommand\section{\@startsection{section}{1}{\z@}%
  {1.5\baselineskip plus 0.2\baselineskip minus 0.1\baselineskip}%
  {0.65\baselineskip}{\normalfont\Large\bfseries}}
\renewcommand\subsection{\@startsection{subsection}{2}{\z@}%
  {1.15\baselineskip plus 0.2\baselineskip minus 0.1\baselineskip}%
  {0.45\baselineskip}{\normalfont\large\bfseries}}
\renewcommand\subsubsection{\@startsection{subsubsection}{3}{\z@}%
  {0.9\baselineskip plus 0.15\baselineskip minus 0.1\baselineskip}%
  {0.35\baselineskip}{\normalfont\normalsize\bfseries}}
\makeatother

\makeatletter
\patchcmd{\@setauthors}{\centering\footnotesize}{\centering\normalsize}{}{}
\patchcmd{\@setauthors}{\MakeUppercase{\authors}}{\authors}{}{}
\patchcmd{\maketitle}{\uppercasenonmath\shorttitle}{}{}{}
\patchcmd{\maketitle}{\@nx\MakeUppercase{\the\toks@}}{\the\toks@}{}{}
\makeatother

\usepackage{hyperref}
\hypersetup{
  hidelinks,pdfencoding=auto,
  pdftitle={Loss of derivatives for hyperbolic operators with coefficients oscillating in time},
  pdfauthor={Wenhui Chen and Michael Reissig},
  pdfkeywords={second-order hyperbolic equations, time oscillations, Lipschitz regularity, energy estimates, loss of derivatives}
}
\usepackage[nameinlink,capitalise,noabbrev]{cleveref}

\theoremstyle{plain}
\newtheorem{theorem}{Theorem}[section]
\newtheorem{lemma}[theorem]{Lemma}
\newtheorem{prop}[theorem]{Proposition}

\newtheorem{hypoth}{Hypothesis}[section]
\theoremstyle{definition}
\newtheorem{defn}{Definition}[section]
\theoremstyle{remark}
\newtheorem{remark}{Remark}[section]

\newcommand{\ml}{\mathcal}
\newcommand{\mb}{\mathbb}
\newcommand{\dd}{\,\mathrm{d}}
\DeclareMathOperator{\divv}{div}

\title[Hyperbolic operators with oscillating coefficients]
  {Loss of derivatives for hyperbolic operators with coefficients oscillating in time}
\author[W. Chen]{Wenhui Chen}
\address{School of Mathematics and Information Science, Guangzhou University,
  Guangzhou 510006, P. R. China}
\email{wenhui.chen.math@gmail.com}
\author[M. Reissig]{Michael Reissig}
\address{Institute of Applied Analysis, Faculty of Mathematics and Computer Science,
  Technical University Bergakademie Freiberg, 09596 Freiberg, Germany}
\email{reissig@math.tu-freiberg.de}
\date{}

\begin{document}
\raggedbottom
\begin{abstract}
We consider second-order strictly hyperbolic operators whose coefficients are Lipschitz continuous in space and may oscillate near the initial time.
Their first two time derivatives are bounded by the corresponding powers of $\frac{1}{t}(\log\frac{1}{t})^\alpha$ with $0\leqslant\alpha<1$, and the first time derivative satisfies the corresponding spatial Lipschitz bound.
By combining Tarama's energy with frequency-dependent weights, we prove an energy estimate without loss for $\alpha=0$ and with a stretched logarithmic upper bound for the loss when $0<\alpha<1$, which implies an arbitrarily small loss in the Sobolev scale.
These estimates extend the analysis of Colombini, Del Santo and Fanelli [\emph{Ann. Scuola Norm. Sup. Pisa Cl. Sci. (5)} \textbf{26} (2025)] for spatially Lipschitz coefficients to time-derivative bounds with an additional logarithmic factor.
\end{abstract}
\maketitle

\section{Introduction}\label{Section-Introduction}
For a fixed $T\in(0,1)$, we consider the strictly hyperbolic equation
\begin{align}\label{Eq-hyperbolic-matrix}
\ml{L}u:=\partial_t^2u-\divv \bigl(A(t,x)\nabla u\bigr)=0 \ \ \text{for}\ \ (t,x)\in(0,T]\times\mb{R}^n,
\end{align}
where the real symmetric coefficient matrix $A(t,x)=\big(a_{kl}(t,x)\big)_{k,l=1}^n$ is defined on the strip $[0,T]\times\mb{R}^n$ and satisfies
\begin{align*}
	\lambda_0\mathrm{I}_n\leqslant A(t,x)\leqslant\Lambda_0\mathrm{I}_n
\end{align*}
for constants $\lambda_0,\Lambda_0>0$, where $\mathrm{I}_n$ is the $n$-dimensional identity matrix.
In component form, the operator $\ml{L}$ in \eqref{Eq-hyperbolic-matrix} is given by
\begin{align}\label{L-operator}
\ml{L}u=\partial_t^2u-\sum\limits_{k,l=1}^n\partial_k\bigl(a_{kl}(t,x)\partial_lu\bigr),
\end{align}
where the coefficients $a_{kl}=a_{kl}(t,x)$ satisfy $a_{kl}=a_{lk}$ for all $k,l\in\{1,\dots,n\}$ and
\begin{align}\label{Property-akl-1}
\lambda_0|\xi|^2\leqslant\sum\limits_{k,l=1}^na_{kl}(t,x)\xi_k\xi_l\leqslant\Lambda_0|\xi|^2
\end{align}
for any $(t,x)\in[0,T]\times\mb{R}^n$ and any $\xi\in\mb{R}^n$, which guarantees strict hyperbolicity.
The aim of this paper is to study the loss of derivatives in energy estimates under Hypothesis \ref{Hypothesis-A}, where the time oscillations of the coefficients $a_{kl}$ are controlled by $\frac{1}{t}(\log\frac{1}{t})^\alpha$ with $0\leqslant\alpha<1$.

The influence of low regularity of the coefficients on solvability of hyperbolic Cauchy problems has been studied since \cite{Colombini-De Giorgi-Spagnolo=1979}.
Weakening the time regularity may produce a loss of derivatives, while spatial dependence requires suitable regularity assumptions, as in \cite{Colombini-Lerner=1995}.
Such a loss may also occur for coefficients which are smooth for $t>0$ but oscillate infinitely often as $t\downarrow0$, so bounds on their time derivatives provide another way to study this problem.

Under suitable spatial smoothness assumptions on the coefficients and their first two time derivatives, continuity at $t=0$, and the bounds
\begin{align*}
t\big|\partial_ta_{kl}(t,x)\big|+t^2\big|\partial_t^2a_{kl}(t,x)\big|\leqslant C,
\end{align*}
well-posedness without loss of derivatives was first obtained in \cite{Yamazaki=1990}.
For coefficients depending only on $t$, well-posedness with a finite loss was proved in \cite{Colombini-DelSanto-Kinoshita=2002} under the first-derivative bound $t|\partial_ta_{kl}(t)|\leqslant C$.
An additional logarithmic factor in the bounds for the first two time derivatives was allowed in \cite{Colombini-DelSanto-Reissig=2003}, where a finite loss of derivatives was also obtained.
We refer to \cite{Hirosawa=2003-01,Hirosawa=2003-02,Hirosawa-Reissig=2004} for related results on oscillations and loss of regularity.
For propagation speeds depending only on time, the relation between singular derivative bounds, continuity assumptions and derivative loss was studied in \cite{Ghisi-Gobbino=2021,Ghisi-Gobbino=2023}, including optimality results. For coefficients depending on both $t$ and $x$, the spatial regularity was weakened to Lipschitz or log-Lipschitz continuity in \cite{Colombini-DelSanto-Fanelli=2023}, while retaining the bounds
\begin{align*}
\big|\partial_ta_{kl}(t,x)\big|\lesssim t^{-1}\ \ \text{and}\ \ \big|\partial_t^2a_{kl}(t,x)\big|\lesssim t^{-2}.
\end{align*}
With the corresponding spatial control of $\partial_ta_{kl}(t,x)$, an energy estimate without loss in $H^{1-\theta}\times H^{-\theta}$ was proved in \cite[Theorem 2.1]{Colombini-DelSanto-Fanelli=2023} for $\theta\in[0,1)$ in the Lipschitz case.
In the log-Lipschitz case, \cite[Theorem 2.2]{Colombini-DelSanto-Fanelli=2023} gives a local in-time estimate with a loss increasing linearly in time for $\theta\in(0,1)$.
The proof in \cite{Colombini-DelSanto-Fanelli=2023} controls the time oscillations under these spatial regularity assumptions by using the lower-order energy correction introduced in \cite{Tarama=2007} together with para-differential calculus with parameters.
This combination of methods was also used in \cite{Colombini-DelSanto-Fanelli-Metivier=2013,Colombini-DelSanto-Fanelli-Metivier=2013-2}.
Coefficients depending only on $t$ with bounds of the form
\begin{align*}
	\big|\partial_t^pa_{kl}(t)\big|\lesssim\Big(\frac{1+|\log t|^\alpha}{t}\Big)^p\ \ \text{for}\ \ p\in\{1,2\}
\end{align*}
are further discussed at the end of \cite[Section 2]{Colombini-DelSanto-Fanelli=2023}.
These bounds give no loss for $\alpha=0$, an arbitrarily small loss for $0<\alpha<1$, and a finite loss for $\alpha=1$. To the best of our knowledge, the corresponding estimates for $0<\alpha<1$ have not been established for coefficients depending on both $t$ and $x$ under spatial Lipschitz assumptions.
Since the logarithmic factor introduces frequency-dependent growth into the localized energy bounds, extending the argument of \cite{Colombini-DelSanto-Fanelli=2023} requires additional control of the weight ratios in the commutator and remainder estimates coupling different dyadic blocks.

To describe the time-derivative bounds, we set
\begin{align*}
h(t;\alpha):=\frac{1}{t}\Big(\log\frac{1}{t}\Big)^{\alpha}.
\end{align*}
\begin{hypoth}\label{Hypothesis-A}
Let $\alpha\in[0,1)$, and let $a_{kl}: [0,T]\times\mb{R}^n\to\mb{R}$ be twice differentiable with respect to $t$ on $(0,T]\times\mb{R}^n$ and satisfy the following assumptions:
\begin{description}
	\item[(H1)] There exists $C_0>0$ such that for all $k,l\in\{1,\dots,n\}$,
	\begin{align*}
	\big|a_{kl}(t,x+y)-a_{kl}(t,x)\big|\leqslant C_0|y|
	\end{align*}
	holds for any $t\in[0,T]$ and any $x,y\in\mb{R}^n$ with $|y|\in(0,1)$.
	\item[(H2)] There exist $C_1,C_2,C_3>0$ such that for all $k,l\in\{1,\dots,n\}$,
	\begin{align*}
	\big|\partial_t^pa_{kl}(t,x)\big|&\leqslant C_p[h(t;\alpha)]^p\ \ \text{with}\ \ p\in\{1,2\},\\
	\big|\partial_ta_{kl}(t,x+y)-\partial_ta_{kl}(t,x)\big|&\leqslant C_3|y|h(t;\alpha)
	\end{align*}
	hold for any $t\in(0,T]$ and any $x,y\in\mb{R}^n$ with $|y|\in(0,1)$.
\end{description}
\end{hypoth}

\begin{defn}[Weighted Sobolev spaces]\label{Definition-Weighted-Sobolev}
For $s\in\mb{R}$, $K>0$ and $\alpha\in[0,1)$, we define the weighted Sobolev space $H^s_{\alpha,K}$ as the space of all $f\in\ml{S}'$ such that
\begin{align}\label{Weighted-Sobolev-Norm}
\|f\|_{H^s_{\alpha,K}}^2:=\sum\limits_{j=0}^{+\infty}\exp\bigl(-K(j+1)^{\alpha}\bigr)2^{2sj}\|\Delta_jf\|_{L^2}^2<+\infty,
\end{align}
where $\Delta_j$ are the standard dyadic blocks defined by \eqref{Definition-Dyadic-Blocks}.
\end{defn}

\begin{itemize}
\item If $\alpha=0$, this norm is equivalent to the classical $H^s$-norm (see \eqref{EQUIV} or \cite[Chapter 2]{Bahouri-Chemin-Danchin=2011}).
\item If $\alpha\in(0,1)$, up to a change of the positive constant, the loss is measured by
\begin{align*}
\exp\bigl(-C(\log\langle D_x\rangle)^\alpha\bigr).
\end{align*}
For related exponentially weighted Sobolev spaces, see \cite{Cicognani-Lorenz=2018}.
\end{itemize}

\begin{theorem}[Energy estimates]\label{Thm-1}
Let $\ml{L}$ be defined by \eqref{L-operator}, and assume that its coefficients $a_{kl}$ satisfy \eqref{Property-akl-1} and Hypothesis \ref{Hypothesis-A}. Fix $\alpha\in[0,1)$ and $\theta\in(0,1)$. Then there exist constants $K,C>0$, depending only on $n$, $\lambda_0$, $\Lambda_0$, $C_0,C_1,C_2,C_3$, $T$, $\alpha$, $\theta$ and the fixed cut-off functions, such that every $u\in\ml{C}^2([0,T],H^{\infty})$ satisfies
\begin{align}\label{Energy-Est}
&\sup\limits_{t\in[0,T]}\Bigl(\|u(t,\cdot)\|_{H^{1-\theta}_{\alpha,K}}+\|\partial_tu(t,\cdot)\|_{H^{-\theta}_{\alpha,K}}\Bigr)\notag\\
&\qquad\leqslant C\Big(\|u(0,\cdot)\|_{H^{1-\theta}}+\|\partial_tu(0,\cdot)\|_{H^{-\theta}}+\int_0^T\|\ml{L}u(\tau,\cdot)\|_{H^{-\theta}}\,\mathrm{d}\tau\Big).
\end{align}
In particular, if $\alpha\in(0,1)$, then for every $\epsilon\in(0,1-\theta)$ there exists $C_{\epsilon}>0$ such that
\begin{align}\label{Energy-Est-Sobolev}
&\sup\limits_{t\in[0,T]}\Bigl(\|u(t,\cdot)\|_{H^{1-\theta-\epsilon}}+\|\partial_tu(t,\cdot)\|_{H^{-\theta-\epsilon}}\Bigr)\notag\\
&\qquad\leqslant C_{\epsilon}\Big(\|u(0,\cdot)\|_{H^{1-\theta}}+\|\partial_tu(0,\cdot)\|_{H^{-\theta}}+\int_0^T\|\ml{L}u(\tau,\cdot)\|_{H^{-\theta}}\,\mathrm{d}\tau\Big).
\end{align}
\end{theorem}

\begin{remark}\label{Remark-Cauchy}
For $0<\alpha<1$, applying the usual regularization and density arguments to \eqref{Energy-Est-Sobolev} gives well-posedness of the Cauchy problem with an arbitrarily small loss of derivatives.
Here, for initial data in $H^{1-\theta}\times H^{-\theta}$ and a forcing term in $L^1((0,T),H^{-\theta})$, the solution belongs to
\begin{align*}
\ml{C}\big([0,T],H^{1-\theta-\epsilon}\big)\cap\ml{C}^1\big([0,T],H^{-\theta-\epsilon}\big)
\end{align*}
for every $0<\epsilon<1-\theta$.
For $\alpha=0$, \eqref{Energy-Est} gives well-posedness without loss, as shown in \cite[Theorem 2.1]{Colombini-DelSanto-Fanelli=2023}.
\end{remark}

\begin{remark}\label{Remark-loss-endpoints}
 For $\alpha=0$, the estimate \eqref{Energy-Est} indicates no loss of derivatives. The endpoint $\theta=0$ in this case follows from \cite[Theorem 2.1]{Colombini-DelSanto-Fanelli=2023}. For $\alpha\in(0,1)$, our estimate gives a stretched logarithmic upper bound for the loss. When we express this bound as a variable Sobolev order, the logarithmic power in that order is $\alpha-1$, since, for $|\xi|>0$,
\begin{align*}
\langle\xi\rangle^{-C(\log\langle\xi\rangle)^{\alpha-1}}
=\exp\bigl(-C(\log\langle\xi\rangle)^\alpha\bigr).
\end{align*}
For coefficients depending only on $t$, finite-loss results for $\alpha=1$ are known from \cite{Colombini-DelSanto-Reissig=2003,Hirosawa=2003-01,Hirosawa=2003-02,Hirosawa-Reissig=2004}.
\end{remark}

\begin{remark}\label{Remark-example}
The assumptions allow coefficients without a limit as $t\downarrow0$.
For example, let $b\in W^{1,\infty}$ be real-valued and nonzero such that $\|b\|_{L^\infty}\leqslant1$, and let $0<\kappa<\frac12$.
For $t>0$, set
\begin{align*}
A(t,x):=\Big[1+\kappa b(x)\sin\Big(\Big(\log\frac{1}{t}\Big)^{1+\alpha}\,\Big)\Big]\mathrm{I}_n,
\end{align*}
and define $A(0,x):=\mathrm{I}_n$.
Because of $\|b\|_{L^\infty}\leqslant1$ and $0<\kappa<\frac12$, we have
\begin{align*}
	(1-\kappa)\mathrm{I}_n\leqslant A(t,x)\leqslant(1+\kappa)\mathrm{I}_n.
\end{align*}
Writing $\psi(t):=(\log\frac{1}{t})^{1+\alpha}$, one may arrive at
\begin{align*}
\psi'(t)&=-\frac{1+\alpha}{t}\Big(\log\frac{1}{t}\Big)^\alpha=-(1+\alpha)h(t;\alpha),\\
\psi''(t)&=\frac{1+\alpha}{t^2}\Big[\Big(\log\frac{1}{t}\Big)^\alpha+\alpha\Big(\log\frac{1}{t}\Big)^{\alpha-1}\,\Big]
\leqslant C[h(t;\alpha)]^2,
\end{align*}
in which the last inequality follows from $\log\frac{1}{t}\geqslant\log\frac{1}{T}>0$ for $t\in(0,T]$, with $C$ depending on $T$ and $\alpha$.
The chain rule shows
\begin{align*}
\partial_tA(t,x)&=\kappa b(x)\psi'(t)\cos\psi(t)\,\mathrm{I}_n,\\
\partial_t^2A(t,x)&=\kappa b(x)\big(\psi''(t)\cos\psi(t)-[\psi'(t)]^2\sin\psi(t)\big)\mathrm{I}_n,
\end{align*}
so $\|b\|_{L^\infty}\leqslant1$ gives
\begin{align*}
\big\|\partial_t^pA(t,\cdot)\big\|_{L^\infty}\leqslant C[h(t;\alpha)]^p\ \ \text{for}\ \ p\in\{1,2\}.
\end{align*}
Moreover, taking the Lipschitz representative of $b\in W^{1,\infty}$, one finds
\begin{align*}
	|b(x+y)-b(x)|\leqslant\|\nabla b\|_{L^\infty}|y|
\end{align*} and, hence,
\begin{align*}
\big|a_{kl}(t,x+y)-a_{kl}(t,x)\big|&\leqslant\kappa\|\nabla b\|_{L^\infty}|y|,\\
\big|\partial_ta_{kl}(t,x+y)-\partial_ta_{kl}(t,x)\big|&\leqslant\kappa\|\nabla b\|_{L^\infty}|y|\,|\psi'(t)|\leqslant C|y|h(t;\alpha).
\end{align*}
These estimates verify (H1)--(H2). For $\alpha>0$, the quantity $t\|\partial_tA(t,\cdot)\|_{L^\infty}$ is unbounded as $t\downarrow0$.
\end{remark}

\medskip\paragraph{Notation.}
The positive constants $c$ and $C$ may change from line to line.
We write $f\lesssim g$ if $f\leqslant Cg$ for some $C>0$, define $f\gtrsim g$ analogously, and use $f\sim g$ when both inequalities hold.
We denote by $\ml{S}$ the Schwartz space, by $\ml{S}'$ its dual, and by $\widehat f$ the Fourier transform of $f$ with respect to $x$.
With $D_x:=-i\nabla_x$ and $\langle\xi\rangle^2:=1+|\xi|^2$, the operator $\langle D_x\rangle$ denotes the Fourier multiplier with symbol $\langle\xi\rangle$.
All spatial norms are taken over $\mb{R}^n$.
For $s\in\mb{R}$, the inhomogeneous Sobolev space $H^s$ is equipped with the norm
\begin{align*}
\|f\|_{H^s}:=\|\langle D_x\rangle^sf\|_{L^2},
\end{align*}
and we put $H^\infty:=\bigcap_{s\in\mb{R}}H^s$.
The Lipschitz norm is given by
\begin{align*}
\|a\|_{\mathrm{Lip}}:=\|a\|_{L^\infty}+\|\nabla a\|_{L^\infty}.
\end{align*}
We use the $L^2$ inner product
\begin{align*}
(f,g)_{L^2}:=\int_{\mb{R}^n}f(x)\overline{g(x)}\,\mathrm{d}x,
\end{align*}
and denote the corresponding operator adjoint by the superscript $*$.
The notation $B_R(0)$ denotes the ball of radius $R>0$ centered at the origin in $\mb{R}^n$, and $\chi_I$ denotes the characteristic function of an interval $I$.

\section{Approximation and regularization}\label{Section-Regularisation}
\subsection{\texorpdfstring{Approximation of coefficients in time}{Approximation of coefficients in time}}\label{Section-Regular-in-time}
Since (H2) allows the coefficients to oscillate as $t\downarrow0$, we approximate $a_{kl}$ by coefficients $a_{\varepsilon,kl}$ which are constant near $t=0$ and agree with $a_{kl}$ away from $0$.

Let $\rho_\eta(t):=\eta^{-1}\rho(\frac{t}{\eta})$, where $\rho\in\ml{C}_c^\infty((-1,1))$ is nonnegative and satisfies $\int_{\mb{R}}\rho(s)\dd s=1$. We also choose a smooth function $\eta_0\in\ml{C}^{\infty}(\mb{R})$ valued in $[0,1]$ such that $\eta_0(s)=1$ if $s\leqslant\frac{1}{4}$ and $\eta_0(s)=0$ if $s\geqslant\frac{3}{4}$, and put $\eta_{\varepsilon}(t):=\eta_0(\frac{t}{3\varepsilon}+\frac{1}{12})$. Then, $\eta_{\varepsilon}(t)=1$ if $t\leqslant\frac{\varepsilon}{2}$ and $\eta_{\varepsilon}(t)=0$ if $t\geqslant2\varepsilon$.

Following \cite{Colombini-DelSanto-Fanelli=2023}, we freeze the coefficients outside $[\varepsilon,T]$, smooth the frozen extension in time, and use a cut-off to retain the original coefficients for $t\geqslant2\varepsilon$.
We define
\begin{align*}
a_{\varepsilon,kl}(t,x):=\widetilde{\widetilde{a}}_{\varepsilon,kl}(t,x)\eta_{\varepsilon}(t)+a_{kl}(t,x)\bigl(1-\eta_{\varepsilon}(t)\bigr)
\end{align*}
for $\varepsilon\in(0,\frac{T}{2}]$, where
\begin{align*}
\widetilde{a}_{\varepsilon,kl}(t,x):=\begin{cases}
a_{kl}(\varepsilon,x)&\text{if}\ \ t\in(-\infty,\varepsilon),\\
a_{kl}(t,x)&\text{if}\ \ t\in[\varepsilon,T],\\
a_{kl}(T,x)&\text{if}\ \ t\in(T,+\infty).
\end{cases}
\end{align*}
The convolution in time is given by
\begin{align*}
\widetilde{\widetilde{a}}_{\varepsilon,kl}(t,x):=\big(\rho_{\frac{\varepsilon}{2}}\ast_{(t)}\widetilde{a}_{\varepsilon,kl}\big)(t,x).
\end{align*}
By the support of $\rho$ and the definition of $\eta_\varepsilon$, the resulting coefficient satisfies $a_{\varepsilon,kl}(t,x)=a_{kl}(t,x)$ for $t\in[2\varepsilon,T]$ and $a_{\varepsilon,kl}(t,x)=a_{kl}(\varepsilon,x)$ for $t\in[0,\frac{\varepsilon}{2}]$.
The regularized coefficients $a_{\varepsilon,kl}$ also satisfy the following properties:
\begin{enumerate}[label=(\roman*)]
	\item The condition
	\begin{align*}
	\lambda_0|\xi|^2\leqslant\sum\limits_{k,l=1}^na_{\varepsilon,kl}(t,x)\xi_k\xi_l\leqslant\Lambda_0|\xi|^2
	\end{align*}
	holds for any $t\in[0,T]$ and any $x,\xi\in\mb{R}^n$.
	\item There exists $C_0>0$ (possibly different from the one appearing in (H1), but still independent of $\varepsilon\in(0,\frac{T}{2}]$) such that for all $k,l\in\{1,\dots,n\}$, the condition
		\begin{align*}
		|a_{\varepsilon,kl}(t,x+y)-a_{\varepsilon,kl}(t,x)|\leqslant C_0|y|
	\end{align*}
	holds for any $t\in[0,T]$ and any $x,y\in\mb{R}^n$ with $|y|\in(0,1)$.
	\item We may choose a bounded nonnegative function $\varphi_\varepsilon$ on $[0,T]$ such that
	\begin{align}\label{VARPHI-Setting}
	\varphi_{\varepsilon}(t)=0\ \ &\text{if}\ \ t\in\Big[0,\frac{\varepsilon}{2}\Big],\notag\\
	0\leqslant\varphi_{\varepsilon}(t)\leqslant C h(t;\alpha)\ \ &\text{if}\ \ t\in\Big[\frac{\varepsilon}{2},T\Big],
	\end{align}
	and there exist $C_1,C_2,C_3>0$, independent of $\varepsilon\in(0,\frac{T}{2}]$, such that
	\begin{align*}
	|\partial_t^pa_{\varepsilon,kl}(t,x)|&\leqslant C_p[\varphi_{\varepsilon}(t)]^p\ \ \text{with}\ \ p\in\{1,2\},\\
	|\partial_ta_{\varepsilon,kl}(t,x+y)-\partial_ta_{\varepsilon,kl}(t,x)|&\leqslant C_3|y|\varphi_{\varepsilon}(t),
	\end{align*}
	hold for any $t\in[0,T]$, any $x,y\in\mb{R}^n$ with $|y|\in(0,1)$, and all $k,l\in\{1,\dots,n\}$.
\end{enumerate}
\begin{proof}
{(i) Since freezing in time, convolution with the nonnegative kernel $\rho_{\frac{\varepsilon}{2}}$, and convex combinations preserve \eqref{Property-akl-1}, the coefficients $a_{\varepsilon,kl}$ satisfy the required ellipticity bounds.}\\
(ii) By (H1) and the normalization of $\rho_{\frac{\varepsilon}{2}}$, we obtain
\begin{align*}
|\widetilde{\widetilde{a}}_{\varepsilon,kl}(t,x+y)-\widetilde{\widetilde{a}}_{\varepsilon,kl}(t,x)|&
\leqslant\int_{\mb{R}}\rho_{\frac{\varepsilon}{2}}(t-\tau)|\widetilde{a}_{\varepsilon,kl}(\tau,x+y)-\widetilde{a}_{\varepsilon,kl}(\tau,x)|\,\mathrm{d}\tau\\
&\leqslant C_0|y|\int_{\mb{R}}\rho_{\frac{\varepsilon}{2}}(t-\tau)\,\mathrm{d}\tau=C_0|y|.
\end{align*}
By applying the triangle inequality to the definition of $a_{\varepsilon,kl}$, we obtain the spatial Lipschitz bound.\\
{(iii) The coefficients are constant for $t\leqslant\frac{\varepsilon}{2}$, so $\partial_ta_{\varepsilon,kl}(t,x)=0$ there. Since the nonconstant part of the convolution involves time variables comparable with $\varepsilon$ on $[\frac{\varepsilon}{2},2\varepsilon]\cap[0,T]$, Hypothesis \ref{Hypothesis-A} shows}
\begin{align}\label{Difference-Regularized-Coefficient}
|\widetilde{\widetilde{a}}_{\varepsilon,kl}(t,x)-a_{kl}(t,x)|
&\leqslant C\varepsilon h(c\varepsilon;\alpha)
\leqslant C\Big(\log\frac{1}{\varepsilon}\Big)^{\alpha},
\end{align}
where $c>0$ is independent of $\varepsilon$. Since $|\partial_t^p\eta_{\varepsilon}(t)|\leqslant C\varepsilon^{-p}$ for $p\in\{1,2\}$, it follows from \eqref{Difference-Regularized-Coefficient} that
\begin{align*}
|\widetilde{\widetilde{a}}_{\varepsilon,kl}-a_{kl}|\,|\partial_t\eta_{\varepsilon}|
&\leqslant C h(t;\alpha),\\
|\widetilde{\widetilde{a}}_{\varepsilon,kl}-a_{kl}|\,|\partial_t^2\eta_{\varepsilon}|
&\leqslant C[h(t;\alpha)]^2
\end{align*}
on this interval. The frozen extension is locally Lipschitz in time and satisfies
\begin{align*}
	\|\partial_t\widetilde a_{\varepsilon,kl}\|_{L^\infty(\mb{R}\times\mb{R}^n)}\leqslant C h(\varepsilon;\alpha).
\end{align*} By differentiating the mollifier in the second-order term, we obtain the following bounds:
\begin{align*}
\|\partial_t\widetilde{\widetilde a}_{\varepsilon,kl}\|_{L^\infty(\mb{R}\times\mb{R}^n)}
&\leqslant C h(\varepsilon;\alpha),\\
\|\partial_t^2\widetilde{\widetilde a}_{\varepsilon,kl}\|_{L^\infty(\mb{R}\times\mb{R}^n)}
&\leqslant \|\rho_{\frac{\varepsilon}{2}}'\|_{L^1(\mb{R})}
\|\partial_t\widetilde a_{\varepsilon,kl}\|_{L^\infty(\mb{R}\times\mb{R}^n)}\leqslant C\varepsilon^{-1}h(\varepsilon;\alpha).
\end{align*}
Since $t$ and $\varepsilon$ are comparable on $[\frac{\varepsilon}{2},2\varepsilon]\cap[0,T]$, these bounds are at most $C h(t;\alpha)$ and $C[h(t;\alpha)]^2$, respectively, with constants depending on $T<1$. The formula
\begin{align*}
\partial_t^2a_{\varepsilon,kl}
&=\partial_t^2\widetilde{\widetilde{a}}_{\varepsilon,kl}\eta_{\varepsilon}
+\partial_t^2a_{kl}(1-\eta_{\varepsilon})+2(\partial_t\widetilde{\widetilde{a}}_{\varepsilon,kl}-\partial_ta_{kl})\partial_t\eta_{\varepsilon}
+(\widetilde{\widetilde{a}}_{\varepsilon,kl}-a_{kl})\partial_t^2\eta_{\varepsilon}
\end{align*}
then proves the asserted bounds for $p\in\{1,2\}$. To check the mixed bound, let $\delta_yb(t,x):=b(t,x+y)-b(t,x)$. Integrating the spatial difference of $\partial_ta_{kl}$ in (H2) gives, in the transition interval,
\begin{align*}
|\delta_y(\widetilde{\widetilde a}_{\varepsilon,kl}-a_{kl})|
&\leqslant C|y|\varepsilon h(c\varepsilon;\alpha),\notag\\
|\delta_y\partial_t\widetilde{\widetilde a}_{\varepsilon,kl}|
&\leqslant C|y|h(c\varepsilon;\alpha).
\end{align*}
By substituting these bounds into the formula for $\partial_ta_{\varepsilon,kl}$, we obtain the required spatial Lipschitz estimate. We may therefore take
\begin{align*}
\varphi_\varepsilon(t):=
\begin{cases}
0&\text{if}\ \ 0\leqslant t\leqslant\frac{\varepsilon}{2},\\
h(t;\alpha)&\text{if} \ \ \frac{\varepsilon}{2}<t\leqslant T,
\end{cases}
\end{align*}
to complete our proof.
\end{proof}

\subsection{Tools from the Littlewood--Paley theory} \label{Sec:Littlewood-Paley theory}
We recall the Littlewood--Paley decomposition and the associated characterization of Sobolev norms from \cite{Bony=1981,Metivier=2008,Bahouri-Chemin-Danchin=2011}.

Let $\chi=\chi(\xi)$ be a smooth radial function supported in $B_2(0)$, equal to $1$ in a neighborhood of $B_1(0)$ and non-increasing in the radial variable.
We set $\phi(\xi):=\chi(\xi)-\chi(2\xi)$ and, for $j\in\mb{N}_0$, define
\begin{align*}
\chi_j(\xi):=\chi(2^{-j}\xi)\ \ \text{and}\ \ \phi_j(\xi):=\phi(2^{-j}\xi).
\end{align*}

We define the dyadic blocks $\Delta_j$ for $j\in\mb{N}_0$ by
\begin{align}\label{Definition-Dyadic-Blocks}
\Delta_ju:=\begin{cases}
\chi(D_x)u&\text{if}\ \ j=0,\\
\phi_j(D_x)u&\text{if}\ \ j\geqslant 1.
\end{cases}
\end{align}
The low-frequency cut-off operator $S_j$ is defined via
\begin{align*}
S_ju:=\chi_j(D_x)u=\sum\limits_{k=0}^{j}\Delta_ku.
\end{align*}
For every $u\in\ml{S}'$, the Littlewood--Paley decomposition
\begin{align*}
u=\sum\limits_{j=0}^{+\infty}\Delta_ju
\end{align*}
holds in the topology of $\ml{S}'$.

For every $s\in\mb{R}$, the dyadic characterization of $H^s$ gives a constant $C_s>0$ such that
\begin{align}\label{EQUIV}
\frac{1}{C_s}\sum\limits_{j=0}^{+\infty}2^{2sj}\|\Delta_ju\|_{L^2}^2\leqslant\|u\|_{H^s}^2\leqslant C_s\sum\limits_{j=0}^{+\infty}2^{2sj}\|\Delta_ju\|_{L^2}^2.
\end{align}

\subsection{Tools from the para-differential calculus with parameters}\label{Subsection-para-diff-parameter}
To treat the low spatial regularity and preserve coercivity, we use para-differential calculus with parameters \cite{Metivier=1986,Metivier-Zumbrun=2005,Metivier=2008}, as in \cite{Colombini-DelSanto-Fanelli-Metivier=2013,Colombini-DelSanto-Fanelli-Metivier=2013-2}.

For a symbol $a=a(x,\xi)$ which is smooth in $\xi$ and bounded in $x$, spatial regularization gives a symbol $\sigma_a=\sigma_a(x,\xi)$ smooth in both variables. We define the associated para-differential operator by
\begin{align*}
\ml{T}_a:=\sigma_a(x,D_x).
\end{align*}
When the symbol also depends on $t$, we treat time as a parameter.

For symbols which are Lipschitz in $x$ and satisfy the usual symbol estimates in $\xi$, different admissible regularizations give operators differing by a remainder of one lower order. In the case $a=a(x)$, one possible para-product is
\begin{align*}
\ml{T}_a:\ u\mapsto\ml{T}_au:=S_0a\,S_3u+\sum\limits_{j=1}^{+\infty}S_ja\,\Delta_{j+3}u,
\end{align*}
which is a para-product with the displayed choice of frequency cut-offs, as in \cite{Bony=1981} and \cite[Chapter 2]{Bahouri-Chemin-Danchin=2011}. For symbols depending also on $\xi$, the products in this formula must be interpreted through symbol quantization, as specified below.

Positivity of the coefficient alone does not give coercivity of the associated para-differential operator. For a bounded scalar function $a=a(x)$ satisfying $a(x)\geqslant c>0$ almost everywhere, one has $\|au\|_{L^2}\geqslant c\|u\|_{L^2}$, but the same lower bound does not hold in general for $\ml{T}_a$.

To obtain coercivity, we use para-differential calculus with a sufficiently large parameter $\gamma\geqslant2$, as in \cite{Metivier=1986,Metivier-Zumbrun=2005}. For a coefficient $a=a(x)$ and $u\in\ml{S}$, we set
\begin{align}\label{T-f-gamma}
\ml{T}_{a}^{\gamma}u:=S_{\mu-1}a\,S_{\mu+2}u+\sum\limits_{m=\mu}^{+\infty}S_ma\,\Delta_{m+3}u,
\end{align}
where $\mu=\mu(\gamma):=\lfloor\log_2\gamma\rfloor\geqslant1$. For a general symbol $a=a(t,x,\xi,\gamma)$, we use the left quantization of
\begin{align}\label{Symbol-Quantization}
\sigma_a^\gamma(t,x,\xi)&:=S_{\mu-1}^{(x)}a(t,x,\xi,\gamma)\chi_{\mu+2}(\xi)+\sum\limits_{m=\mu}^{+\infty}
S_m^{(x)}a(t,x,\xi,\gamma)\phi_{m+3}(\xi),\notag\\
\ml{T}_a^\gamma&:=\sigma_a^\gamma(x,D_x).
\end{align}
Here, $S_m^{(x)}$ acts only on $x$, and \eqref{Symbol-Quantization} reduces to \eqref{T-f-gamma} when $a$ is independent of $\xi$. We require the symbol estimates
\begin{align}\label{Lip-Symbol-Class}
\|\partial_\xi^\beta a(t,\cdot,\xi,\gamma)\|_{L^\infty}
+\|\nabla_x\partial_\xi^\beta a(t,\cdot,\xi,\gamma)\|_{L^\infty}
\leqslant C_\beta(\gamma+|\xi|)^{d-|\beta|}
\end{align}
for a symbol of order $d$, with spatial derivatives understood in the weak sense. For a real symbol of order zero satisfying $a\geqslant c>0$, the parameter-dependent positivity estimate gives, for all sufficiently large $\gamma$,
\begin{align*}
\mathrm{Re}\,(\ml{T}_a^\gamma u,u)_{L^2}
\geqslant \frac{c}{2}\|u\|_{L^2}^2\ \ \mbox{and}\ \
\|\ml{T}_a^\gamma u\|_{L^2}\geqslant\frac{c}{2}\|u\|_{L^2}.
\end{align*}
Since the lower threshold for $\gamma$ depends only on $c$ and finitely many symbol bounds in \eqref{Lip-Symbol-Class}, we may choose it uniformly for a family with uniform bounds.

For fixed $\gamma$, the Sobolev spaces $H^s_\gamma$ defined through the parameter-dependent Littlewood--Paley decomposition coincide with $H^s$, with equivalence constants depending on $\gamma$, as shown in \cite[Section 3.2]{Colombini-DelSanto-Fanelli-Metivier=2013}.

\subsection{Symbolic calculus}\label{Sub-Sec=symbolic calculus}
We introduce the order-zero symbol associated with the regularized principal part of $\ml{L}$:
\begin{align*}
\varpi_{\varepsilon}:=(\gamma^2+|\xi|^2)^{-\frac{1}{2}}\Big(\gamma^2+\sum\limits_{k,l=1}^na_{\varepsilon,kl}(t,x)\xi_k\xi_l\Big)^{\frac{1}{2}}.
\end{align*}
To account for the logarithmic factor in (H2), we include the factor $(j+1)^\alpha$ in the frequency-dependent regularization used in \cite{Colombini-De Giorgi-Spagnolo=1979,Colombini-Lerner=1995}.
For each dyadic block, we choose the time regularization parameter $J=J_j$ by
\begin{align}\label{Definition-J}
J_j:=\min\Big\{(j+1)^{\alpha}2^{-j-4},\frac{T}{4}\Big\}\ \ \text{for}\ \ j\in\mb{N}_0.
\end{align}
With this choice, $J_j\in(0,\frac{T}{2}]$ and we may set $\varepsilon=J_j$ in the time regularization. For all sufficiently large $j$, we have $J_j=(j+1)^\alpha2^{-j-4}$, while the finitely many remaining blocks only affect the constants.

By applying the symbol regularization as in \cite{Colombini-DelSanto-Fanelli-Metivier=2013,Colombini-DelSanto-Fanelli-Metivier=2013-2}, we obtain the following bounds for $\varpi_J$ and its time derivatives. Suppressing the superscript $\gamma$, we have, for $\beta_0,\beta_1\in\mb{N}_0^n$ with $|\beta_1|\geqslant1$,
\begin{align*}
|\partial_{\xi}^{\beta_0}\sigma_{\varpi_J}(t,x,\xi,\gamma)|&\lesssim (\gamma+|\xi|)^{-|\beta_0|},\\
|\partial_x^{\beta_1}\partial_{\xi}^{\beta_0}\sigma_{\varpi_J}(t,x,\xi,\gamma)|&\lesssim (\gamma+|\xi|)^{-|\beta_0|+|\beta_1|-1},\\
|\partial_{\xi}^{\beta_0}\sigma_{\partial_t\varpi_J}(t,x,\xi,\gamma)|&\lesssim \varphi_J(t)\,(\gamma+|\xi|)^{-|\beta_0|},\\
|\partial_x^{\beta_1}\partial_{\xi}^{\beta_0}\sigma_{\partial_t\varpi_J}(t,x,\xi,\gamma)|&\lesssim \varphi_J(t)\,(\gamma+|\xi|)^{-|\beta_0|+|\beta_1|-1},\\
|\partial_{\xi}^{\beta_0}\sigma_{\partial_t^2\varpi_J}(t,x,\xi,\gamma)|&\lesssim [\varphi_J(t)]^2(\gamma+|\xi|)^{-|\beta_0|},\\
|\partial_x^{\beta_1}\partial_{\xi}^{\beta_0}\sigma_{\partial_t^2\varpi_J}(t,x,\xi,\gamma)|&\lesssim [\varphi_J(t)]^2(\gamma+|\xi|)^{-|\beta_0|+|\beta_1|},
\end{align*}
for $t\in[0,T]$, $\gamma\geqslant2$ and $x,\xi\in\mb{R}^n$. The constants may depend on $\beta_0,\beta_1$ and the structural bounds, but not on $j,t,x,\xi,\gamma$. The last line follows from the $L^\infty_x$ bound on $\partial_t^2a_{J,kl}$.

Since the symbols $\varpi_J^{\pm\frac12}$ have uniformly positive lower bounds and uniform spatial Lipschitz seminorms, we may fix $\gamma$ independently of $j$ and $t$ so that
\begin{align*}
\Bigl\|\ml{T}^{\gamma}_{\varpi_J^{-\frac{1}{2}}}g(t,\cdot)\Bigr\|_{L^2}\geqslant C\|g(t,\cdot)\|_{L^2}\ \ \text{and}\ \ \Bigl\|\ml{T}^{\gamma}_{\varpi_J^{\frac{1}{2}}(\gamma^2+|\xi|^2)^{\frac{1}{2}}}g(t,\cdot)\Bigr\|_{L^2}\geqslant C\|g(t,\cdot)\|_{H^1}.
\end{align*}
Since right composition with a Fourier multiplier commutes with the symbol regularization in \eqref{Symbol-Quantization}, we obtain the second inequality by applying the order-zero coercivity estimate for $\varpi_J^{\frac12}$ to $(\gamma^2+|D_x|^2)^{\frac12}g$. All subsequent constants may depend on this fixed $\gamma$.

We use the following Lipschitz symbolic calculus from \cite[Chapters 5--6]{Metivier=2008}, with corresponding log-Lipschitz statements in \cite[Proposition 3.19]{Colombini-DelSanto-Fanelli-Metivier=2013}.
\begin{prop}[Lipschitz symbolic calculus]\label{Prop-01}
If $a$ satisfies the $L^\infty_x$ part of \eqref{Lip-Symbol-Class}, then $\ml{T}_a^\gamma$ has order $d$ as follows:
\begin{align*}
\|\ml{T}_a^\gamma v\|_{H^s_\gamma}\leqslant C\|v\|_{H^{s+d}_\gamma}.
\end{align*}
 If $a$ and $b$ satisfy the full condition \eqref{Lip-Symbol-Class}, with orders $d$ and $e$, respectively, then
\begin{align*}
\ml{T}^{\gamma}_a\ml{T}^{\gamma}_b
&=\ml{T}^{\gamma}_{ab}+\ml{R}^{\gamma}_{a,b},\notag\\
(\ml{T}^{\gamma}_a)^*
&=\ml{T}^{\gamma}_{\overline a}+\widetilde{\ml{R}}^{\gamma}_a,
\end{align*}
where $\ml{R}^{\gamma}_{a,b}$ has order $d+e-1$ and $\widetilde{\ml{R}}^{\gamma}_a$ has order $d-1$. In particular, for every $s\in\mb{R}$,
\begin{align*}
\|\ml{R}^{\gamma}_{a,b}v\|_{H^s_\gamma}
&\leqslant C\|v\|_{H^{s+d+e-1}_\gamma},\\
\|\widetilde{\ml{R}}^{\gamma}_av\|_{H^s_\gamma}
&\leqslant C\|v\|_{H^{s+d-1}_\gamma}.
\end{align*}
Here, $\|v\|_{H^s_\gamma}:=\|(\gamma^2+|D_x|^2)^{\frac{s}{2}}v\|_{L^2}$, and $C$ depends on finitely many symbol seminorms. If $a$ and $b$ are smooth functions of $\varpi_J$ of order zero, the two remainders have order $-1$. The estimates also apply to their products with powers of $(\gamma^2+|\xi|^2)^{\frac12}$. If one factor is $\partial_t f(\varpi_J)$, the corresponding bound contains one factor $\varphi_J(t)$.
\end{prop}

\section{Energy estimates}\label{Section-Energy-Estimates}
To prove Theorem \ref{Thm-1}, we use the localized energy of \cite{Colombini-DelSanto-Fanelli=2023} with the scale \eqref{Definition-J} and the weights introduced in Subsection \ref{Sec:weight-comparison}.
For a given $u\in\ml{C}^2([0,T],H^{\infty})$, we first para-linearize $\ml{L}$ as follows:
\begin{align}\label{Eq-01}
\partial_t^2u=\sum\limits_{k,l=1}^n\partial_k\bigl(a_{kl}(t,x)\,\partial_{l}u\bigr)+\ml{L}u=\sum\limits_{k,l=1}^n\partial_k\bigl(\ml{T}^{\gamma}_{a_{kl}}\partial_lu\bigr)+\widetilde{\ml{L}}u,
\end{align}
where $\ml{T}^{\gamma}_{a_{kl}}$ is the para-product defined in \eqref{T-f-gamma} and
\begin{align}\label{Defn-L}
\widetilde{\ml{L}}u:=\ml{L}u+\sum\limits_{k,l=1}^n\partial_k\bigl((a_{kl}-\ml{T}^{\gamma}_{a_{kl}})\,\partial_lu\bigr).
\end{align}
By applying $\Delta_j$ to \eqref{Eq-01}, we find that each block $u_j:=\Delta_ju$, for $j\in\mb{N}_0$, satisfies
\begin{align}\label{Eq-02}
\partial_t^2u_{j}=\sum\limits_{k,l=1}^n\partial_k\bigl(\ml{T}^{\gamma}_{a_{kl}}\partial_lu_{j}\bigr)+\sum\limits_{k,l=1}^n\partial_k\bigl(\bigl[\Delta_{j},\ml{T}^{\gamma}_{a_{kl}}\bigr]\partial_lu\bigr)+(\widetilde{\ml{L}}u)_{j},
\end{align}
in which $(\widetilde{\ml{L}}u)_j:=\Delta_j(\widetilde{\ml{L}}u)$ and $[A,B]:=AB-BA$.

\subsection{The localized Tarama energy} \label{Sec:Tarama energy}
For the localized energy, we define
\begin{align*}
v_j:=\ml{T}^{\gamma}_{\varpi_J^{-\frac{1}{2}}}\partial_t u_j-\ml{T}^{\gamma}_{\partial_t(\varpi_J^{-\frac{1}{2}})}u_j\ \ \text{and}\ \ w_j:=\ml{T}^{\gamma}_{\varpi_J^{\frac{1}{2}}(\gamma^2+|\xi|^2)^{\frac{1}{2}}}u_j.
\end{align*}
Along with the classical energy
\begin{align*}
\mb{E}_j^{\mathrm{clas}}(t,u):=\|\partial_tu_j(t,\cdot)\|_{L^2}^2+\|\nabla u_j(t,\cdot)\|_{L^2}^2+\|u_j(t,\cdot)\|_{L^2}^2,
\end{align*}
we use the following modified energy introduced in \cite{Tarama=2007}:
\begin{align*}
\mb{E}_j^{\mathrm{Tara}}(t,u):=\|v_j(t,\cdot)\|_{L^2}^2+\|w_j(t,\cdot)\|_{L^2}^2+\|u_j(t,\cdot)\|_{L^2}^2.
\end{align*}
The lower-order correction in $v_j$ cancels the leading terms containing a single time derivative of the principal symbol when we differentiate the energy.

By \eqref{VARPHI-Setting}, we have $\varphi_J(t)=0$ for $t\leqslant\frac{J}{2}$. For $t\geqslant\frac{J}{2}$ and sufficiently large $j$, the choice \eqref{Definition-J} gives
\begin{align*}
\varphi_J(t)\lesssim h\Big(\frac{J}{2};\alpha\Big)
&\lesssim (j+1)^{-\alpha}2^j
\big[\log\big((j+1)^{-\alpha}2^{j+5}\big)\big]^{\alpha}\\
&\lesssim2^j.
\end{align*}
By including the finitely many low-frequency blocks in the constant, we obtain
\begin{align}\label{Boundedness-Crucial}
0\leqslant2^{-j}\varphi_J(t)\leqslant C
\end{align}
uniformly for $t\in[0,T]$ and $j\in\mb{N}_0$.
By using coercivity and the symbolic calculus from Subsection \ref{Sub-Sec=symbolic calculus}, we obtain
\begin{align}\label{dt-uj-control}
\|\partial_tu_j(t,\cdot)\|_{L^2}\lesssim \Bigl\|\ml{T}^{\gamma}_{\varpi_J^{-\frac{1}{2}}}\partial_tu_j(t,\cdot)\Bigr\|_{L^2}&\lesssim\|v_j(t,\cdot)\|_{L^2}+\Bigl\|\ml{T}^{\gamma}_{\partial_t(\varpi_J^{-\frac{1}{2}})}u_j(t,\cdot)\Bigr\|_{L^2}\notag\\
&\lesssim\|v_j(t,\cdot)\|_{L^2}+\varphi_{J}(t)\|u_j(t,\cdot)\|_{L^2}\notag\\
&\lesssim \bigl[\mb{E}_j^{\mathrm{Tara}}(t,u)\bigr]^{\frac{1}{2}},
\end{align}
where we used \eqref{Boundedness-Crucial} together with
\begin{align}\label{Eq-03}
2^j\|u_j(t,\cdot)\|_{L^2}\sim\|u_j(t,\cdot)\|_{H^1}\lesssim\Bigl\|\ml{T}^{\gamma}_{\varpi_J^{\frac{1}{2}}(\gamma^2+|\xi|^2)^{\frac{1}{2}}}u_j(t,\cdot)\Bigr\|_{L^2}=\|w_j(t,\cdot)\|_{L^2}.
\end{align}
For $j=0$, the equivalence in \eqref{Eq-03} follows from the ball support of {$\widehat{\Delta_0u}$} and the inhomogeneous $H^1$ norm.
The last two estimates imply
\begin{align*}
	\mb{E}_j^{\mathrm{clas}}(t,u)\lesssim\mb{E}_j^{\mathrm{Tara}}(t,u).
\end{align*} Conversely, the uniform boundedness of the symbols, \eqref{Boundedness-Crucial}, and the spectral localization of $u_j$ give immediately
\begin{align*}
\|v_j(t,\cdot)\|_{L^2}&\lesssim\|\partial_tu_j(t,\cdot)\|_{L^2}+\varphi_J(t)\|u_j(t,\cdot)\|_{L^2}
\lesssim\bigl[\mb{E}_j^{\mathrm{clas}}(t,u)\bigr]^{\frac{1}{2}},\\
\|w_j(t,\cdot)\|_{L^2}&\lesssim\|u_j(t,\cdot)\|_{H^1}
\lesssim\bigl[\mb{E}_j^{\mathrm{clas}}(t,u)\bigr]^{\frac{1}{2}}.
\end{align*}
Consequently,
\begin{align}\label{Equivalence-Tarama-Classical}
\mb{E}_j^{\mathrm{Tara}}(t,u)\sim\mb{E}_j^{\mathrm{clas}}(t,u)
\end{align}
uniformly for $t\in[0,T]$ and $j\in\mb{N}_0$.

\subsection{\texorpdfstring{First-order time derivative of the localized Tarama energy}{First-order time derivative of the localized Tarama energy}}
By differentiating the squared $L^2$ norm of $u_j$, we obtain
\begin{align}\label{Energy-01}
\frac{\mathrm{d}}{\mathrm{d}t}\|u_j(t,\cdot)\|_{L^2}^2&=2\mathrm{Re}\,\bigl(u_j(t,\cdot),\partial_tu_j(t,\cdot)\bigr)_{L^2}\notag\\
&\leqslant\|u_j(t,\cdot)\|_{L^2}^2+\|\partial_tu_j(t,\cdot)\|_{L^2}^2\notag\\
&\lesssim \mb{E}_j^{\mathrm{Tara}}(t,u),
\end{align}
where we used \eqref{dt-uj-control}. It remains to estimate the terms involving $v_j$ and $w_j$.

\subsubsection{\texorpdfstring{Estimates for the term involving $v_j$}{Estimates for the term involving vj}}
By differentiating $v_j$ in time and substituting \eqref{Eq-02}, we obtain

\begin{align}\label{Energy-02}
&\frac{\mathrm{d}}{\mathrm{d}t}\|v_j(t,\cdot)\|_{L^2}^2
=2\mathrm{Re}\,\big(v_j(t,\cdot),\partial_tv_j(t,\cdot)\big)_{L^2}\notag\\
&\qquad=-2\mathrm{Re}\,\Big(v_j(t,\cdot),\ml{T}^{\gamma}_{\partial_t^2(\varpi_J^{-\frac{1}{2}})}u_j(t,\cdot)\Big)_{L^2}+2\mathrm{Re}\,\Big(v_j(t,\cdot),\ml{T}^{\gamma}_{\varpi_J^{-\frac{1}{2}}}\sum\limits_{k,l=1}^n\partial_k\big(\ml{T}^{\gamma}_{a_{kl}}\partial_lu_j(t,\cdot)\big)\Big)_{L^2}\notag\\
&\qquad\quad+2\mathrm{Re}\,\Big(v_j(t,\cdot),\sum\limits_{k,l=1}^n\ml{T}^{\gamma}_{\varpi_J^{-\frac{1}{2}}}\partial_k\big(\big[\Delta_j,\ml{T}^{\gamma}_{a_{kl}}\big]\partial_lu(t,\cdot)\big)\Big)_{L^2}\notag\\
&\qquad\quad+2\mathrm{Re}\,\Big(v_j(t,\cdot),\ml{T}^{\gamma}_{\varpi_J^{-\frac{1}{2}}}\big(\widetilde{\ml{L}}u(t,\cdot)\big)_j\Big)_{L^2},
\end{align}
where we employed the identity
\begin{align*}
\partial_tv_j=-\ml{T}^{\gamma}_{\partial_t^2(\varpi_J^{-\frac{1}{2}})}u_j+\ml{T}^{\gamma}_{\varpi_J^{-\frac{1}{2}}}\partial_t^2u_j.
\end{align*}

By the $L^\infty_x$ symbol bounds for $\partial_t^2(\varpi_J^{-\frac12})$, Proposition \ref{Prop-01} and \eqref{Eq-03}, we obtain
\begin{align*}
\Big|\mathrm{Re}\,\Big(v_j(t,\cdot),\ml{T}^{\gamma}_{\partial_t^2(\varpi_J^{-\frac{1}{2}})}u_j(t,\cdot)\Big)_{L^2}\Big|&\lesssim [\varphi_J(t)]^2\|v_j(t,\cdot)\|_{L^2}\|u_j(t,\cdot)\|_{L^2}\\
&\lesssim 2^{-j}[\varphi_J(t)]^2\|v_j(t,\cdot)\|_{L^2}\|w_j(t,\cdot)\|_{L^2}\\
&\lesssim 2^{-j}[\varphi_J(t)]^2\,\mb{E}_j^{\mathrm{Tara}}(t,u).
\end{align*}
By the uniform $L^2$ boundedness of $\ml{T}^{\gamma}_{\varpi_J^{-\frac{1}{2}}}$, we also obtain
\begin{align*}
\Big|\mathrm{Re}\,\Big(v_j(t,\cdot),\ml{T}^{\gamma}_{\varpi_J^{-\frac{1}{2}}}\big(\widetilde{\ml{L}}u(t,\cdot)\big)_{j}\Big)_{L^2}\Big|\lesssim \big[\mb{E}_j^{\mathrm{Tara}}(t,u)\big]^{\frac{1}{2}}\big\|\big(\widetilde{\ml{L}}u(t,\cdot)\big)_{j}\big\|_{L^2}.
\end{align*}

By the definition of $\ml{T}_a^\gamma$, we write
\begin{align}\label{Eq-04}
\sum\limits_{k,l=1}^n\partial_k\big(\ml{T}^{\gamma}_{a_{kl}}\partial_l u_j\big)=\sum\limits_{k,l=1}^n\ml{T}^{\gamma}_{\partial_ka_{kl}}\partial_l u_j-\ml{T}^{\gamma}_{\sum_{k,l=1}^{n}a_{kl}\xi_k\xi_l}u_j,
\end{align}
{where}
\begin{align*}
\ml{T}^{\gamma}_{\partial_ka_{kl}}\partial_l u_j=S_{\mu-1}\partial_ka_{kl}\,S_{\mu+2}\partial_l u_j+\sum\limits_{\nu=\mu+3}^{+\infty}S_{\nu-3}\partial_ka_{kl}\,\Delta_\nu\partial_l u_j
\end{align*}
with a fixed integer $\mu=\lfloor\log_2\gamma\rfloor$. By spectral localization and the uniform spatial Lipschitz bounds for $a_{kl}$, we obtain, as in \cite[Section 4.2.4]{Colombini-DelSanto-Fanelli-Metivier=2013},
\begin{align}\label{Eq-05}
\big\|\ml{T}^{\gamma}_{\partial_ka_{kl}}\partial_l u_j(t,\cdot)\big\|_{L^2}&\leqslant \|S_{\mu-1}\partial_ka_{kl}(t,\cdot)\|_{L^{\infty}}\|S_{\mu+2}\partial_lu_j(t,\cdot)\|_{L^2}\notag\\
&\quad+\sum\limits_{\nu=\mu+3}^{+\infty}\|\nabla S_{\nu-3}a_{kl}(t,\cdot)\|_{L^{\infty}}\|\Delta_{\nu}\partial_l u_j(t,\cdot)\|_{L^2}\notag\\
&\lesssim \big[\mb{E}_j^{\mathrm{Tara}}(t,u)\big]^{\frac{1}{2}},
\end{align}
where the implicit constant depends on the uniform spatial Lipschitz norm of $a_{kl}$ and the fixed parameter $\gamma$. The remaining principal term will be combined with the derivative of $\|w_j(t,\cdot)\|_{L^2}^2$.
\subsubsection{\texorpdfstring{A dyadic commutator estimate}{A dyadic commutator estimate}} \label{Sec:commutator term}
We denote the commutator term in \eqref{Eq-02} by
\begin{align*}
\ml{B}_j(t,u):=\sum\limits_{k,l=1}^n\partial_k\big(\big[\Delta_j,\ml{T}^{\gamma}_{a_{kl}}\big]\partial_lu(t,\cdot)\big).
\end{align*}
Its contribution to \eqref{Energy-02} is
\begin{align}\label{Definition-Q-Commutator}
Q^{\gamma}_{\mathrm{Com},j}(t,u):=
2\mathrm{Re}\,\Big(v_j(t,\cdot),
\ml{T}^{\gamma}_{\varpi_J^{-\frac{1}{2}}}\ml{B}_j(t,u)\Big)_{L^2}.
\end{align}

\begin{lemma}[Dyadic commutator estimate]\label{Lemma-Raw-Commutator}
There exist an integer $N_0\geqslant1$ and nonnegative constants $c_{jm}$, uniformly bounded with respect to $j,m$, such that
\begin{align}\label{Raw-Commutator}
\|\ml{B}_j(t,u)\|_{L^2}
\leqslant C\sum\limits_{m=0}^{+\infty}c_{jm}2^m\|u_m(t,\cdot)\|_{L^2},
\end{align}
where $c_{jm}=0$ unless $|j-m|\leqslant N_0$. The integer $N_0$ is chosen large enough to include the fixed low-frequency range determined by $\gamma$. Both $N_0$ and $C$ may depend on $\gamma$ and the fixed cut-offs, and $C$ also depends on the uniform spatial Lipschitz bound, but neither depends on $t,j,m$ or $u$.
\end{lemma}
\begin{proof}
Recalling \eqref{T-f-gamma}, we write
\begin{align*}
\big[\Delta_j,\ml{T}^{\gamma}_{a_{kl}}\big]w
&=[\Delta_j,S_{\mu-1}a_{kl}]S_{\mu+2}w
+\sum\limits_{m=\mu}^{+\infty}[\Delta_j,S_ma_{kl}]\Delta_{m+3}w,
\end{align*}
where $\mu=\lfloor\log_2\gamma\rfloor$ is fixed. The first term contains only a fixed number of low-frequency blocks. For the second term, spectral localization gives
\begin{align*}
[\Delta_j,S_ma_{kl}]\partial_lu_{m+3}\equiv0
\ \ \text{if}\ \ |j-(m+3)|>N_0.
\end{align*}
For a Lipschitz function $b$, the convolution kernel $K_j$ of $\Delta_j$ gives
\begin{align*}
[\Delta_j,b]z(x)
&=\int_{\mb{R}^n}K_j(x-y)\big(b(y)-b(x)\big)z(y)\,\mathrm{d}y,\\
\|[\Delta_j,b]z\|_{L^2}
&\leqslant C2^{-j}\|\nabla b\|_{L^\infty}\|z\|_{L^2}.
\end{align*}
By applying Bernstein's inequality on the spectral support of the commutator, we obtain, for the nonzero interactions,
\begin{align*}
\big\|\partial_k\big([\Delta_j,S_ma_{kl}]\partial_lu_{m+3}(t,\cdot)\big)\big\|_{L^2}
\lesssim2^m\|a_{kl}(t,\cdot)\|_{\mathrm{Lip}}\|u_{m+3}(t,\cdot)\|_{L^2}.
\end{align*}
Summing over the finitely many admissible indices proves \eqref{Raw-Commutator}.
\end{proof}

\subsubsection{\texorpdfstring{Estimates for the term involving $w_j$}{Estimates for the term involving wj}}
By differentiating the definition of $w_j$, we obtain
\begin{align*}
\frac{\mathrm{d}}{\mathrm{d}t}\|w_j(t,\cdot)\|_{L^2}^2=2\mathrm{Re}\,\Big(w_j(t,\cdot),\ml{T}^{\gamma}_{\partial_t(\varpi_J^{\frac{1}{2}})(\gamma^2+|\xi|^2)^{\frac{1}{2}}}u_j(t,\cdot)+\ml{T}^{\gamma}_{\varpi_J^{\frac{1}{2}}(\gamma^2+|\xi|^2)^{\frac{1}{2}}}\partial_tu_j(t,\cdot)\Big)_{L^2}.
\end{align*}

Thanks to
\begin{align*}
\partial_t(\varpi_J^{\frac12})=\frac12\varpi_J^{-\frac12}\partial_t\varpi_J=\varpi_J\partial_t(-\varpi_J^{-\frac12}),
\end{align*}
Proposition \ref{Prop-01} gives
\begin{align*}
&\mathrm{Re}\,\Big(w_j(t,\cdot),\ml{T}^{\gamma}_{\partial_t(\varpi_J^{\frac{1}{2}})(\gamma^2+|\xi|^2)^{\frac{1}{2}}}u_j(t,\cdot)\Big)_{L^2}\\
&\qquad=\mathrm{Re}\,\Big(w_j(t,\cdot),\ml{T}^{\gamma}_{\varpi_J(\gamma^2+|\xi|^2)^{\frac{1}{2}}}\ml{T}^{\gamma}_{\partial_t(-\varpi_J^{-\frac{1}{2}})}u_j(t,\cdot)\Big)_{L^2}+\mathrm{Re}\,\big(w_j(t,\cdot),\ml{R}_1^{\gamma}u_j(t,\cdot)\big)_{L^2}.
\end{align*}
Here, the spatial Lipschitz bound for $\partial_ta_{J,kl}$ obtained from (H2) in Subsection \ref{Section-Regular-in-time} allows us to apply Proposition \ref{Prop-01} to symbols of orders one and zero, which gives an order-zero remainder satisfying
\begin{align*}
\|\ml{R}_1^\gamma z\|_{L^2}\leqslant C\varphi_J(t)\|z\|_{L^2}.
\end{align*}
Combining \eqref{Eq-03} with \eqref{Boundedness-Crucial}, we obtain
\begin{align*}
\big|\mathrm{Re}\,\big(w_j(t,\cdot),\ml{R}_1^{\gamma}u_j(t,\cdot)\big)_{L^2}\big|&\lesssim\varphi_J(t)\|w_j(t,\cdot)\|_{L^2}\|u_j(t,\cdot)\|_{L^2}\\
&\lesssim\mb{E}_{j}^{\mathrm{Tara}}(t,u).
\end{align*}
By applying the same composition estimate to the term containing $\partial_tu_j$, we obtain
\begin{align*}
&\mathrm{Re}\,\Big(w_j(t,\cdot),\ml{T}^{\gamma}_{\varpi_J^{\frac{1}{2}}(\gamma^2+|\xi|^2)^{\frac{1}{2}}}\partial_tu_j(t,\cdot)\Big)_{L^2}\\
&\qquad=\mathrm{Re}\,\Big(w_j(t,\cdot),\ml{T}^{\gamma}_{\varpi_J(\gamma^2+|\xi|^2)^{\frac{1}{2}}}\ml{T}^{\gamma}_{\varpi_J^{-\frac{1}{2}}}\partial_tu_j(t,\cdot)\Big)_{L^2}+\mathrm{Re}\,\big(w_j(t,\cdot),\ml{R}_2^{\gamma}\partial_tu_j(t,\cdot)\big)_{L^2},
\end{align*}
where the order-zero remainder satisfies
\begin{align*}
\|\ml{R}_2^\gamma z\|_{L^2}\leqslant C\|z\|_{L^2}.
\end{align*}
By this bound and \eqref{dt-uj-control}, we obtain
\begin{align*}
\big|\mathrm{Re}\,\big(w_j(t,\cdot),\ml{R}_2^{\gamma}\partial_tu_j(t,\cdot)\big)_{L^2}\big|&\lesssim \|w_j(t,\cdot)\|_{L^2}\|\partial_t u_j(t,\cdot)\|_{L^2}\\
&\lesssim\mb{E}_j^{\mathrm{Tara}}(t,u).
\end{align*}
Combining these identities, we obtain
\begin{align*}
\frac{\mathrm{d}}{\mathrm{d}t}\|w_j(t,\cdot)\|_{L^2}^2&=2\mathrm{Re}\,\Big(w_j(t,\cdot),\ml{T}^{\gamma}_{\varpi_J(\gamma^2+|\xi|^2)^{\frac{1}{2}}}\Big(\ml{T}^{\gamma}_{\varpi_J^{-\frac{1}{2}}}\partial_tu_j(t,\cdot)-\ml{T}^{\gamma}_{\partial_t(\varpi_J^{-\frac{1}{2}})}u_j(t,\cdot)\Big)\Big)_{L^2}\\
&\quad+2\mathrm{Re}\,\big(w_j(t,\cdot),\ml{R}_1^{\gamma}u_j(t,\cdot)+\ml{R}_2^{\gamma}\partial_tu_j(t,\cdot)\big)_{L^2}\\
&=:2\mathrm{Re}\,\Big(w_j(t,\cdot),\ml{T}^{\gamma}_{\varpi_J(\gamma^2+|\xi|^2)^{\frac{1}{2}}}v_j(t,\cdot)\Big)_{L^2}+Q_{1,j}^{\gamma}(t,u),
\end{align*}
where
\begin{align*}
|Q_{1,j}^{\gamma}(t,u)|\lesssim \mb{E}_j^{\mathrm{Tara}}(t,u).
\end{align*}

By applying the adjoint and composition estimates in Proposition \ref{Prop-01}, we may arrive at
\begin{align*}
&\mathrm{Re}\,\Big(w_j(t,\cdot),\ml{T}^{\gamma}_{\varpi_J(\gamma^2+|\xi|^2)^{\frac{1}{2}}}v_j(t,\cdot)\Big)_{L^2}\\
&\qquad=\mathrm{Re}\,\Big(\ml{T}^{\gamma}_{\varpi_J(\gamma^2+|\xi|^2)^{\frac{1}{2}}}w_j(t,\cdot),v_j(t,\cdot)\Big)_{L^2}+\mathrm{Re}\,\big(\ml{R}^{\gamma}_3w_j(t,\cdot),v_j(t,\cdot) \big)_{L^2}\\
&\qquad=\mathrm{Re}\,\Big(\ml{T}^{\gamma}_{\varpi_J^{-\frac{1}{2}}}\ml{T}^{\gamma}_{\varpi_J^{\frac{3}{2}}(\gamma^2+|\xi|^2)^{\frac{1}{2}}}w_j(t,\cdot),v_j(t,\cdot)\Big)_{L^2}+\mathrm{Re}\,\big((\ml{R}_3^{\gamma}+\ml{R}_4^{\gamma})w_j(t,\cdot),v_j(t,\cdot)\big)_{L^2},
\end{align*}
where $\ml{R}_3^\gamma$ is the adjoint remainder for a symbol of order one and $\ml{R}_4^\gamma$ is the composition remainder for symbols of orders zero and one. By Proposition \ref{Prop-01}, both are uniformly bounded on $L^2$, which gives
\begin{align*}
\big|\mathrm{Re}\,\big((\ml{R}_3^{\gamma}+\ml{R}_4^{\gamma})w_j(t,\cdot),v_j(t,\cdot)\big)_{L^2}\big|&\lesssim\|w_j(t,\cdot)\|_{L^2}\|v_j(t,\cdot)\|_{L^2}\\
&\lesssim\mb{E}_{j}^{\mathrm{Tara}}(t,u).
\end{align*}
By substituting the definition of $w_j$ and applying Proposition \ref{Prop-01}, we obtain
\begin{align*}
&\mathrm{Re}\,\Big(\ml{T}^{\gamma}_{\varpi_J^{-\frac{1}{2}}}\ml{T}^{\gamma}_{\varpi_J^{\frac{3}{2}}(\gamma^2+|\xi|^2)^{\frac{1}{2}}}w_j(t,\cdot),v_j(t,\cdot)\Big)_{L^2}\\
&\qquad=\mathrm{Re}\,\Big( \ml{T}^{\gamma}_{\varpi_J^{-\frac{1}{2}}}\ml{T}^{\gamma}_{\varpi_J^{\frac{3}{2}}(\gamma^2+|\xi|^2)^{\frac{1}{2}}}\ml{T}^{\gamma}_{\varpi_J^{\frac{1}{2}}(\gamma^2+|\xi|^2)^{\frac{1}{2}}}u_j(t,\cdot), v_j(t,\cdot) \Big)_{L^2}\\
&\qquad=\mathrm{Re}\,\Big(\ml{T}^{\gamma}_{\varpi_J^{-\frac{1}{2}}}\ml{T}^{\gamma}_{\varpi_J^2(\gamma^2+|\xi|^2)}u_j(t,\cdot),v_j(t,\cdot)\Big)_{L^2}+\mathrm{Re}\,\Big(\ml{T}^{\gamma}_{\varpi_J^{-\frac{1}{2}}}\ml{R}_5^{\gamma}u_j(t,\cdot),v_j(t,\cdot)\Big)_{L^2},
\end{align*}
where the composition of two operators of order one gives a remainder $\ml{R}_5^\gamma$ of order one. For the fixed parameter $\gamma$, we obtain
\begin{align*}
\|\ml{R}_5^\gamma u_j(t,\cdot)\|_{L^2}\leqslant C2^j\|u_j(t,\cdot)\|_{L^2},
\end{align*}
which gives
\begin{align*}
\Big|\mathrm{Re}\,\Big(\ml{T}^{\gamma}_{\varpi_J^{-\frac{1}{2}}}\ml{R}_5^{\gamma}u_j(t,\cdot),v_j(t,\cdot)\Big)_{L^2} \Big|&\lesssim 2^j\|u_j(t,\cdot)\|_{L^2}\|v_j(t,\cdot)\|_{L^2}\\&\lesssim \mb{E}_{j}^{\mathrm{Tara}}(t,u).
\end{align*}
Combining these estimates, one derives
\begin{align*}
\frac{\mathrm{d}}{\mathrm{d}t}\|w_j(t,\cdot)\|_{L^2}^2
=2\mathrm{Re}\,\Big(\ml{T}^{\gamma}_{\varpi_J^{-\frac{1}{2}}}\ml{T}^{\gamma}_{\varpi_J^2(\gamma^2+|\xi|^2)}u_j(t,\cdot),v_j(t,\cdot)\Big)_{L^2}
+Q_{2,j}^{\gamma}(t,u),
\end{align*}
where the remainder term
\begin{align*}
Q_{2,j}^{\gamma}(t,u)
&:=Q_{1,j}^{\gamma}(t,u)
+2\mathrm{Re}\,\big((\ml{R}_3^{\gamma}+\ml{R}_4^{\gamma})w_j(t,\cdot),v_j(t,\cdot)\big)_{L^2}\\
&\ \quad+2\mathrm{Re}\,\Big(\ml{T}^{\gamma}_{\varpi_J^{-\frac{1}{2}}}\ml{R}_5^{\gamma}u_j(t,\cdot),v_j(t,\cdot)\Big)_{L^2}
\end{align*}
satisfies
\begin{align*}
	|Q_{2,j}^{\gamma}(t,u)|\lesssim \mb{E}_j^{\mathrm{Tara}}(t,u).
\end{align*}

By the construction in Subsection \ref{Section-Regular-in-time}, we have $a_{J,kl}-a_{kl}=0$ for $t\geqslant2J$ and $|a_{J,kl}-a_{kl}|\leqslant C$ for $t\leqslant2J$. Combining these bounds with \eqref{Eq-04} and \eqref{Eq-05}, we obtain

\begin{align}\label{Energy-03}
&\Big|\mathrm{Re}\,\Big(v_j(t,\cdot),\ml{T}^{\gamma}_{\varpi_J^{-\frac{1}{2}}}\sum\limits_{k,l=1}^n\partial_k\big(\ml{T}^{\gamma}_{a_{kl}}\partial_lu_j(t,\cdot)\big)\Big)_{L^2}+\mathrm{Re}\,\Big(\ml{T}^{\gamma}_{\varpi_J^{-\frac{1}{2}}}\ml{T}^{\gamma}_{\varpi_J^2(\gamma^2+|\xi|^2)}u_j(t,\cdot),v_j(t,\cdot)\Big)_{L^2}\Big|\notag\\
&\lesssim \|v_j(t,\cdot)\|_{L^2}\Big\|\sum\limits_{k,l=1}^n\ml{T}^{\gamma}_{\partial_ka_{kl}}\partial_lu_j(t,\cdot)\Big\|_{L^2}+\|v_j(t,\cdot)\|_{L^2}\Big\|\Big(\ml{T}^{\gamma}_{\varpi_J^2(\gamma^2+|\xi|^2)}-\ml{T}^{\gamma}_{\sum_{k,l=1}^{n}a_{kl}\xi_k\xi_l}\Big)u_j(t,\cdot)\Big\|_{L^2}\notag\\
&\lesssim \|v_j(t,\cdot)\|_{L^2}\big[\mb{E}_j^{\mathrm{Tara}}(t,u)\big]^{\frac{1}{2}}+\|v_j(t,\cdot)\|_{L^2}\Big\|\ml{T}^{\gamma}_{\gamma^2+\sum_{k,l=1}^n(a_{J,kl}-a_{kl})\xi_k\xi_l}u_j(t,\cdot)\Big\|_{L^2}\notag\\
&\lesssim
\mb{E}_j^{\mathrm{Tara}}(t,u)+2^{2j}\chi_{[0,2J]}(t)\|v_j(t,\cdot)\|_{L^2}\|u_j(t,\cdot)\|_{L^2}\notag\\
&\lesssim \big(1+2^j\chi_{[0,2J]}(t)\big)\mb{E}_j^{\mathrm{Tara}}(t,u),
\end{align}
where the identity
\begin{align*}
\varpi_J^2(\gamma^2+|\xi|^2)=\gamma^2+\sum\limits_{k,l=1}^na_{J,kl}\,\xi_k\xi_l
\end{align*}
was used.
\subsection{Estimate of the localized total energy} \label{Sec:localized total energy}
To combine the localized estimates, we define
\begin{align}\label{Definition-qj}
q_j(t):=2^{-j}[\varphi_J(t)]^2+2^j\chi_{[0,2J]}(t)
\end{align}
and the para-differential remainder
\begin{align*}
\ml{P}u:=\sum\limits_{k,l=1}^n
\partial_k\big(\big(a_{kl}-\ml{T}^{\gamma}_{a_{kl}}\big)\partial_lu\big).
\end{align*}
Then, \eqref{Defn-L}, \eqref{Energy-01}, \eqref{Energy-02}, \eqref{Energy-03}, and \eqref{Definition-Q-Commutator} imply
\begin{align}\label{Localized-Energy-Inequality}
\frac{\mathrm{d}}{\mathrm{d}t}\mb{E}^{\mathrm{Tara}}_j(t,u)
&\leqslant C\big(1+q_j(t)\big)\mb{E}^{\mathrm{Tara}}_j(t,u)
+|Q^{\gamma}_{\mathrm{Com},j}(t,u)|\notag\\
&\quad+C\big[\mb{E}^{\mathrm{Tara}}_j(t,u)\big]^{\frac{1}{2}}
\big(\|(\ml{L}u(t,\cdot))_j\|_{L^2}+\|(\ml{P}u(t,\cdot))_j\|_{L^2}\big).
\end{align}
To sum these estimates, we need weights which absorb $q_j(t)$ in \eqref{Definition-qj} and allow comparison between dyadic levels with an integrable time-dependent factor.

\subsubsection{\texorpdfstring{Construction and comparison of the weights}{Construction and comparison of the weights}}\label{Sec:weight-comparison}
By \eqref{VARPHI-Setting} and \eqref{Definition-J}, we have
\begin{align*}
	q_j(t)\leqslant C(j+1)^\alpha\frac{J_j}{(t+J_j)^2}.
\end{align*}
To absorb this term in \eqref{Localized-Energy-Inequality}, we define the phase
\begin{align}\label{Definition-Gj}
G_j(t):=M(j+1)^\alpha\int_0^t\frac{J_j}{(\tau+J_j)^2}\,\mathrm{d}\tau
=M(j+1)^\alpha\frac{t}{t+J_j},
\end{align}
where $M>0$ will be chosen in Lemma \ref{Lemma-Absorption}.
The corresponding energy weight is
\begin{align*}
\omega_j(t):=\exp\Big(-\frac{1}{2}G_j(t)\Big).
\end{align*}

\begin{lemma}[Absorption estimate]\label{Lemma-Absorption}
The constant $M>0$, independent of $j$ and $t$, can be chosen such that
\begin{align}\label{Absorption-qj}
G_j'(t)\geqslant Cq_j(t)
\end{align}
for all $t\in[0,T]$ and $j\in\mb{N}_0$, where $C>0$ is the constant needed in \eqref{Localized-Energy-Inequality}. Moreover,
\begin{align}\label{Size-Gj}
0\leqslant G_j(t)\leqslant M(j+1)^{\alpha}.
\end{align}
\end{lemma}
\begin{proof}
A direct computation gives
\begin{align*}
G_j'(t)=M(j+1)^{\alpha}\frac{J_j}{(t+J_j)^2}.
\end{align*}
We first consider all sufficiently large $j$, for which $J_j=(j+1)^{\alpha}2^{-j-4}$. If $t\leqslant2J_j$, then
\begin{align*}
G_j'(t)\geqslant\frac{M}{9}(j+1)^{\alpha}J_j^{-1}
\gtrsim M2^j,
\end{align*}
which controls the second term in \eqref{Definition-qj}. On the other hand, $\varphi_J(t)=0$ for $t\leqslant\frac{J_j}{2}$. If $t\geqslant\frac{J_j}{2}$, then
\begin{align*}
\log\frac{1}{t}\leqslant C(j+1)
\end{align*}
and, according to \eqref{VARPHI-Setting}, one arrives at
\begin{align*}
2^{-j}[\varphi_J(t)]^2
&\lesssim2^{-j}t^{-2}(j+1)^{2\alpha}\\
&\lesssim (j+1)^{\alpha}\frac{J_j}{(t+J_j)^2}
=M^{-1}G_j'(t).
\end{align*}
By increasing $M$ to include both contributions and the finitely many low-frequency blocks, we obtain \eqref{Absorption-qj}. The bound \eqref{Size-Gj} follows from \eqref{Definition-Gj}.
\end{proof}

For the comparison of the weights, let us put
\begin{align*}
\ell(t):=1+\log\frac{1}{t}\ \ \text{for}\ \ t\in(0,T].
\end{align*}

\begin{lemma}[Phase comparison]\label{Lemma-Phase-Comparison}
For the fixed constant $M$ in Lemma \ref{Lemma-Absorption}, there exists a constant $C>0$, possibly depending on $M$, such that
\begin{align}\label{Phase-Comparison}
|G_j(t)-G_m(t)|
\leqslant C\big([\ell(t)]^{\alpha}+(1+|j-m|)^{\alpha}\big)
\end{align}
for all $t\in(0,T]$ and $j,m\in\mb{N}_0$.
\end{lemma}
\begin{proof}
It suffices to consider small $t>0$, since the remaining time interval only affects the constant. Let $N=N(t)$ be the transition index determined by
\begin{align*}
J_N\leqslant t<J_{N-1}.
\end{align*}
For sufficiently large $N$, the inequality $t<J_{N-1}=N^\alpha2^{-N-3}$ gives
\begin{align*}
(N+3)\log2
&<\log\frac{1}{t}+\alpha\log N\\
&\leqslant\log\frac{1}{t}+\frac{N\log2}{2}+C.
\end{align*}
The same inequality also gives
\begin{align*}
t2^N<2^NJ_{N-1}=\frac{N^\alpha}{8}\leqslant C(N+1)^\alpha.
\end{align*}
Hence, after including the finitely many remaining indices in the constant, we obtain
\begin{align}\label{Properties-N}
N\leqslant C\ell(t)
\ \ \text{and}\ \
t2^N\leqslant C(N+1)^{\alpha}.
\end{align}
If $k<N$, then
\begin{align*}
G_k(t)\leqslant M(k+1)^{\alpha}\frac{t}{J_k}
\lesssim t2^k
\lesssim(N+1)^{\alpha}.
\end{align*}
If $k\geqslant N$, then
\begin{align*}
|G_k(t)-M(k+1)^{\alpha}|
&=M(k+1)^{\alpha}\frac{J_k}{t+J_k}
\lesssim\frac{(k+1)^{2\alpha}2^{-k}}{t}
\lesssim(N+1)^{\alpha},
\end{align*}
where we used $t\geqslant J_N$ and the exponential decay of $2^{N-k}$. For $\alpha\in(0,1)$, recalling the concavity inequality
\begin{align*}
|(j+1)^{\alpha}-(m+1)^{\alpha}|
\leqslant|j-m|^{\alpha},
\end{align*}
we distinguish three cases.
\begin{itemize}
	\item If $j,m<N$, both phases are bounded by $C(N+1)^\alpha$.
	\item If $j,m\geqslant N$, subtracting $M(j+1)^\alpha$ and $M(m+1)^\alpha$ gives the same error bound plus $M|j-m|^\alpha$.
	\item If $j<N\leqslant m$, then
	\begin{align*}
		(m+1)^\alpha
		\leqslant(N+1)^\alpha+(m-N)^\alpha
		\leqslant(N+1)^\alpha+|m-j|^\alpha.
	\end{align*}
The opposite mixed case is identical.
\end{itemize}
 Combining these bounds, we obtain
\begin{align*}
|G_j(t)-G_m(t)|
\leqslant C\big((N+1)^{\alpha}+(1+|j-m|)^{\alpha}\big).
\end{align*}
For $\alpha=0$, the same bound follows directly from $(k+1)^{\alpha}=1$.
Combining it with \eqref{Properties-N}, we obtain \eqref{Phase-Comparison}.
\end{proof}

By Lemma \ref{Lemma-Phase-Comparison}, we may choose $C>0$ in
\begin{align*}
\rho_{\alpha}(t):=C\exp\big(C[\ell(t)]^{\alpha}\big)
\end{align*}
such that
\begin{align}\label{Weight-Ratio}
\frac{\omega_j(t)}{\omega_m(t)}
&=\exp\Big(\frac{G_m(t)-G_j(t)}{2}\Big)\notag\\
&=\exp\Big(\frac{M}{2}\Big((m+1)^\alpha\frac{t}{t+J_m}-(j+1)^\alpha\frac{t}{t+J_j}\Big)\Big)\notag\\
&\leqslant\rho_{\alpha}(t)\exp\big(C(1+|j-m|)^{\alpha}\big).
\end{align}
Since $\alpha<1$, the change of variables $s=\log\frac{1}{t}$ gives
\begin{align}\label{Integrability-rho}
\int_0^T\rho_{\alpha}(t)\,\mathrm{d}t
\lesssim\int_{\log\frac{1}{T}}^{+\infty}
\exp\big(C(1+s)^{\alpha}-s\big)\,\mathrm{d}s<+\infty,
\end{align}
which justifies the integrability of $\rho_{\alpha}$.

\subsubsection{Weighted estimates for the coupled dyadic terms} \label{Sec:coupled dyadic terms}
For $s\in\mb{R}$, we introduce the time-dependent norm
\begin{align*}
\|f\|_{H^s_{\omega(t)}}^2:=
\sum\limits_{j=0}^{+\infty}[\omega_j(t)]^2\,2^{2sj}
\|f_j\|_{L^2}^2,
\end{align*}
where $f_j:=\Delta_jf$. By \eqref{Size-Gj} and \eqref{Weighted-Sobolev-Norm} with $K=M$, we have
\begin{align*}
\|f\|_{H^s_{\alpha,M}}\leqslant\|f\|_{H^s_{\omega(t)}}\lesssim\|f\|_{H^s}
\end{align*}
uniformly for $t\in[0,T]$, and $G_j(0)=0$ gives $\|f\|_{H^s_{\omega(0)}}\sim\|f\|_{H^s}$.
The time dependence of $[\omega_j(t)]^2$ allows us to absorb $q_j(t)$ when differentiating the energy, while $H^s_{\alpha,M}$ measures the loss in Theorem \ref{Thm-1}.
We also introduce the total energy
\begin{align}\label{Definition-Total-Energy}
\ml{E}_{\theta}(t,u):=
\sum\limits_{j=0}^{+\infty}[\omega_j(t)]^2\,2^{-2j\theta}
\mb{E}^{\mathrm{Tara}}_j(t,u),
\end{align}
which follows the total energy construction in \cite[Section 4.4]{Colombini-DelSanto-Fanelli=2023}, with weights adapted to the additional logarithmic factor in (H2).
\begin{lemma}[Weighted dyadic estimates]\label{Lemma-Weighted-Couplings}
Let $\alpha\in[0,1)$ and $\theta\in(0,1)$. For every $t\in(0,T]$, one has
\begin{align}\label{Weighted-Commutator}
\sum\limits_{j=0}^{+\infty}[\omega_j(t)]^2\,2^{-2j\theta}
|Q^{\gamma}_{\mathrm{Com},j}(t,u)|
\leqslant C\rho_{\alpha}(t)\ml{E}_{\theta}(t,u).
\end{align}
Moreover,
\begin{align}\label{Weighted-P}
\|\ml{P}u(t,\cdot)\|_{H^{-\theta}_{\omega(t)}}
\leqslant C\rho_{\alpha}(t)
\|u(t,\cdot)\|_{H^{1-\theta}_{\omega(t)}}.
\end{align}
\end{lemma}
\begin{proof}
By combining Lemma \ref{Lemma-Raw-Commutator} with the uniform boundedness of
$\ml{T}^{\gamma}_{\varpi_J^{-\frac{1}{2}}}$ and \eqref{Equivalence-Tarama-Classical}, we deduce
\begin{align*}
|Q^{\gamma}_{\mathrm{Com},j}(t,u)|
\lesssim\big[\mb{E}^{\mathrm{Tara}}_j(t,u)\big]^{\frac{1}{2}}
\sum\limits_{m=0}^{+\infty}c_{jm}
\big[\mb{E}^{\mathrm{Tara}}_m(t,u)\big]^{\frac{1}{2}}.
\end{align*}
Let us set
\begin{align*}
X_j(t):=\omega_j(t)2^{-j\theta}
\big[\mb{E}^{\mathrm{Tara}}_j(t,u)\big]^{\frac{1}{2}}.
\end{align*}
Since $c_{jm}=0$ for $|j-m|>N_0$, \eqref{Weight-Ratio} gives
\begin{align*}
\sum\limits_{j=0}^{+\infty}[\omega_j(t)]^2\,2^{-2j\theta}
|Q^{\gamma}_{\mathrm{Com},j}(t,u)|&\lesssim\sum\limits_{j,m=0}^{+\infty}c_{jm}
\frac{\omega_j(t)}{\omega_m(t)}2^{\theta(m-j)}X_j(t)X_m(t)\\
&\lesssim\rho_{\alpha}(t)
\sum\limits_{j,m=0}^{+\infty}c_{jm}X_j(t)X_m(t)\\
&\lesssim\rho_{\alpha}(t)\ml{E}_{\theta}(t,u),
\end{align*}
which proves \eqref{Weighted-Commutator}.

In order to prove \eqref{Weighted-P}, we expand the coefficient and $u$ into dyadic blocks and first establish
\begin{align}\label{Dyadic-P}
2^{-j\theta}\|(\ml{P}u)_j(t,\cdot)\|_{L^2}
\leqslant C\sum\limits_{m=0}^{+\infty}
2^{-\delta_{\theta}|j-m|}
\,2^{m(1-\theta)}\|u_m(t,\cdot)\|_{L^2},
\end{align}
with
\begin{align*}
	\delta_{\theta}:=\min\{\theta,1-\theta\}.
\end{align*}
Fixing $k,l$, we set $a:=a_{kl}$ and $a_r:=\Delta_ra$. By \eqref{T-f-gamma}, with $\mu=\lfloor\log_2\gamma\rfloor\geqslant1$, we have
\begin{align}\label{Exact-Paraproduct-Remainder}
(a-\ml{T}_a^\gamma)\partial_lu
&=\sum\limits_{m=0}^{\mu+2}\sum\limits_{r=\mu}^{+\infty}
a_r\partial_lu_m
+\sum\limits_{m=\mu+3}^{+\infty}\sum\limits_{r=m-2}^{+\infty}
a_r\partial_lu_m.
\end{align}
The first sum contains the full high-frequency coefficient tail generated by the low-frequency term in \eqref{T-f-gamma}.
Since every pair in these sums satisfies $r\geqslant1$ and $r\geqslant m-2$, the condition (H1) gives
\begin{align*}
\|a_r(t,\cdot)\|_{L^\infty}\leqslant C2^{-r}\|a(t,\cdot)\|_{\mathrm{Lip}}
\ \ \text{for}\ \  r\geqslant1.
\end{align*}
We separate the two possible frequency interactions.
\begin{itemize}
	\item If $r\geqslant m+4$, then $\Delta_j(a_r\partial_lu_m)$ vanishes unless $|j-r|\leqslant N_1$, where $N_1$ is fixed. Thus,
	\begin{align}\label{High-Coefficient-Low-u}
		2^{-j\theta}\|\Delta_j\partial_k(a_r\partial_lu_m)(t,\cdot)\|_{L^2}
		&\leqslant C2^{j(1-\theta)}2^{-r}2^m\|u_m(t,\cdot)\|_{L^2}\notag\\
		&\leqslant C2^{-\theta(j-m)}
		2^{m(1-\theta)}\|u_m(t,\cdot)\|_{L^2}.
	\end{align}
	For each $j$, only a bounded number of such $r$ occur.
	\item If $m-2\leqslant r\leqslant m+3$, the product has Fourier support in a ball of radius $C2^m$, so only $j\leqslant m+N_1$ can occur. In this case,
	\begin{align}\label{Comparable-Frequencies}
		2^{-j\theta}\|\Delta_j\partial_k(a_r\partial_lu_m)(t,\cdot)\|_{L^2}
		&\leqslant C2^{j(1-\theta)}\|u_m(t,\cdot)\|_{L^2}\notag\\
		&=C2^{-(1-\theta)(m-j)}
		2^{m(1-\theta)}\|u_m(t,\cdot)\|_{L^2}.
	\end{align}
\end{itemize}
For each $m$, there are at most six values of $r$ in the second case. By summing \eqref{High-Coefficient-Low-u} and \eqref{Comparable-Frequencies} and absorbing the fixed shifts into the constant, we obtain \eqref{Dyadic-P}.

By multiplying \eqref{Dyadic-P} by $\omega_j(t)$ and using \eqref{Weight-Ratio}, we obtain
\begin{align*}
&\omega_j(t)2^{-j\theta}\|(\ml{P}u)_j(t,\cdot)\|_{L^2}\\
&\qquad\leqslant C\rho_{\alpha}(t)
\sum\limits_{m=0}^{+\infty}
2^{-\delta_{\theta}|j-m|}
\exp\big(C(1+|j-m|)^{\alpha}\big)
\omega_m(t)2^{m(1-\theta)}\|u_m(t,\cdot)\|_{L^2}.
\end{align*}
Thanks to $\alpha<1$, the sequence
\begin{align*}
\big(2^{-\delta_{\theta}|r|}
\exp(C(1+|r|)^{\alpha})\big)_{r\in\mb{Z}}
\end{align*}
belongs to $\ell^1(\mb{Z})$, so Young's inequality for sequences gives \eqref{Weighted-P}.
\end{proof}

\subsubsection{Conclusion of the energy estimate} \label{Sec: final energy estimate}
By \eqref{Equivalence-Tarama-Classical} and $u\in\ml{C}^2([0,T],H^\infty)$, the summands in \eqref{Definition-Total-Energy} have a summable majorant independent of $t$.
Since each regularized coefficient is constant near $t=0$, dominated convergence gives continuity of the total energy at $0$.
On $[\delta,T]$ with $\delta>0$, the time derivatives of the symbols and the phases are uniformly bounded in $j$, so the spatial regularity of $u$ also justifies termwise differentiation.

By differentiating \eqref{Definition-Total-Energy}, applying \eqref{Localized-Energy-Inequality}
and choosing $M$ as in Lemma \ref{Lemma-Absorption}, we absorb the $q_j$ terms
with $-G_j'(t)$ and obtain
\begin{align*}
\frac{\mathrm{d}}{\mathrm{d}t}\ml{E}_{\theta}(t,u)
&\leqslant C\ml{E}_{\theta}(t,u)
+\sum\limits_{j=0}^{+\infty}[\omega_j(t)]^2\,2^{-2j\theta}
|Q^{\gamma}_{\mathrm{Com},j}(t,u)|\\
&\quad+C[\ml{E}_{\theta}(t,u)]^{\frac{1}{2}}
\Big(\|\ml{L}u(t,\cdot)\|_{H^{-\theta}_{\omega(t)}}
+\|\ml{P}u(t,\cdot)\|_{H^{-\theta}_{\omega(t)}}\Big).
\end{align*}
The equivalence \eqref{Equivalence-Tarama-Classical} gives
\begin{align*}
\|u(t,\cdot)\|_{H^{1-\theta}_{\omega(t)}}
+\|\partial_tu(t,\cdot)\|_{H^{-\theta}_{\omega(t)}}
\lesssim[\ml{E}_{\theta}(t,u)]^{\frac{1}{2}}.
\end{align*}
By this bound, Lemma \ref{Lemma-Weighted-Couplings} and $0<\omega_j(t)\leqslant1$, we derive
\begin{align*}
\frac{\mathrm{d}}{\mathrm{d}t}\ml{E}_{\theta}(t,u)
\leqslant C\big(1+\rho_{\alpha}(t)\big)\ml{E}_{\theta}(t,u)
+C[\ml{E}_{\theta}(t,u)]^{\frac{1}{2}}
\|\ml{L}u(t,\cdot)\|_{H^{-\theta}}.
\end{align*}

By applying Gronwall's inequality to $(\ml{E}_\theta+\eta)^{\frac12}$ on $[\delta,T]$ and letting $\eta\downarrow0$ and $\delta\downarrow0$, we can deduce from \eqref{Integrability-rho} that
\begin{align}\label{Bound-Total-Energy}
\sup\limits_{t\in[0,T]}[\ml{E}_{\theta}(t,u)]^{\frac{1}{2}}
\leqslant C\Big(
[\ml{E}_{\theta}(0,u)]^{\frac{1}{2}}
+\int_0^T\|\ml{L}u(\tau,\cdot)\|_{H^{-\theta}}\,
\mathrm{d}\tau\Big).
\end{align}
Since $G_j(0)=0$ by \eqref{Definition-Gj}, we obtain from
\eqref{Equivalence-Tarama-Classical} and the Littlewood--Paley characterization
\begin{align*}
[\ml{E}_{\theta}(0,u)]^{\frac{1}{2}}
\lesssim
\|u(0,\cdot)\|_{H^{1-\theta}}
+\|\partial_tu(0,\cdot)\|_{H^{-\theta}}.
\end{align*}
By \eqref{Size-Gj}, we also have
\begin{align*}
[\omega_j(t)]^2\geqslant\exp\big(-M(j+1)^{\alpha}\big).
\end{align*}
Therefore, \eqref{Bound-Total-Energy} proves \eqref{Energy-Est} with $K=M$.

Finally, for any $\epsilon>0$ and $\alpha\in(0,1)$, one has
\begin{align*}
2^{-2\epsilon j}
\leqslant C_{\epsilon}\exp\big(-M(j+1)^{\alpha}\big)
\ \ \text{for all}\ \ j\in\mb{N}_0.
\end{align*}
It follows that
\begin{align*}
\|f\|_{H^{s-\epsilon}}\leqslant C_{\epsilon}\|f\|_{H^s_{\alpha,M}},
\end{align*}
which completes the proof of \eqref{Energy-Est-Sobolev} and Theorem \ref{Thm-1}.

\section{Final remarks}\label{Sec:concluding remark}
Under (H1)--(H2), Theorem \ref{Thm-1} gives an energy estimate with no loss for $\alpha=0$ and a stretched logarithmic upper bound for the loss when $0<\alpha<1$.
For $\alpha=0$, replacing spatial Lipschitz continuity by log-Lipschitz continuity leads to the estimate with a linearly increasing loss in \cite[Theorem 2.2]{Colombini-DelSanto-Fanelli=2023}.
For coefficients depending only on $t$, faster oscillations give the estimates with a fixed loss discussed in \cite{Colombini-DelSanto-Reissig=2003,Hirosawa-Reissig=2004}.

\begin{enumerate}
\item For $\alpha=1$, a fixed finite loss of derivatives may occur even for coefficients depending only on $t$.
For spatially Lipschitz coefficients, such a loss may lower the regularity of $\nabla u$ beyond the range in which multiplication by Lipschitz functions is bounded.
It remains to determine whether an energy estimate with a finite loss can be obtained under (H1)--(H2) with $\alpha=1$ in a suitable range of Sobolev indices.
\item We propose to replace (H1)--(H2) by the following spatially log-Lipschitz assumptions, while retaining the symmetry and strict hyperbolicity condition \eqref{Property-akl-1}:
\begin{hypoth}\label{Hypothesis-B}
Let $\alpha\in[0,1)$, and let $a_{kl}: [0,T]\times\mb{R}^n\to\mb{R}$ be twice differentiable with respect to $t$ on $(0,T]\times\mb{R}^n$ and satisfy the following assumptions:
\begin{description}
	\item[(H1$'$)] There exists $C_0>0$ such that for all $k,l\in\{1,\dots,n\}$,
	\begin{align*}
	|a_{kl}(t,x+y)-a_{kl}(t,x)|\leqslant C_0|y|\log\Big(1+\frac{1}{|y|}\Big)
	\end{align*}
	holds for any $t\in[0,T]$ and any $x,y\in\mb{R}^n$ with $|y|\in(0,1)$.
	\item[(H2$'$)] There exist $C_1,C_2,C_3>0$ such that for all $k,l\in\{1,\dots,n\}$,
	\begin{align*}
	|\partial_t^pa_{kl}(t,x)|&\leqslant C_p[h(t;\alpha)]^p\ \ \text{with}\ \ p\in\{1,2\},\\
	|\partial_ta_{kl}(t,x+y)-\partial_ta_{kl}(t,x)|&\leqslant C_3|y|\log\Big(1+\frac{1}{|y|}\Big)h(t;\alpha)
	\end{align*}
	hold for any $t\in(0,T]$ and any $x,y\in\mb{R}^n$ with $|y|\in(0,1)$.
\end{description}
\end{hypoth}
For a fixed $\theta\in(0,1)$, we ask whether there exist $\beta^\ast,K,C>0$ and $T_0\in(0,T]$ with $\theta+\beta^\ast T_0<1$ such that every $u\in\ml{C}^2([0,T_0],H^\infty)$ satisfies
\begin{align}\label{Expected-LogLip-Energy}
&\sup\limits_{t\in[0,T_0]}\Big(\|u(t,\cdot)\|_{H^{1-\theta-\beta^{\ast}t}_{\alpha,K}}+\|\partial_tu(t,\cdot)\|_{H^{-\theta-\beta^{\ast}t}_{\alpha,K}}\Big)\notag\\
&\qquad\leqslant C\Big(\|u(0,\cdot)\|_{H^{1-\theta}}+\|\partial_tu(0,\cdot)\|_{H^{-\theta}}
+\int_0^{T_0}\|\ml{L}u(\tau,\cdot)\|_{H^{-\theta-\beta^{\ast}\tau}}\,\mathrm{d}\tau\Big).
\end{align}
For $\alpha=0$, the estimate \eqref{Expected-LogLip-Energy} follows from \cite[Theorem 2.2]{Colombini-DelSanto-Fanelli=2023}.
\item For coefficients depending only on $t$, we consider the Cauchy problem
\begin{align*}
\partial_t^2u-\sum\limits_{k,l=1}^n a_{kl}(t)\partial_k\partial_lu=0,\ \  u(0,x)=u_0(x),\ \  {\partial_tu(0,x)}=u_1(x).
\end{align*}
Under the time derivative bounds in (H2), the known estimates give an arbitrarily small loss for $0<\alpha<1$ and a finite loss for $\alpha=1$, with the relation to optimal regularity discussed in \cite{Colombini-DelSanto-Reissig=2003}.
We also consider the purely spatial problem
\begin{align*}
\partial_t^2u-\sum\limits_{k,l=1}^n \partial_k \big(a_{kl}(x)\partial_lu\big)=0,\ \  u(0,x)=u_0(x),\ \  {\partial_tu(0,x)}=u_1(x),
\end{align*}
with a bounded, symmetric, uniformly positive coefficient matrix satisfying \begin{align*}
	|a_{kl}(x+y)-a_{kl}(x)|\leqslant C_0|y|\Big(\log\Big(1+\frac{1}{|y|}\Big)\Big)^\beta
	\end{align*}
for $\beta\in(0,1]$.
For the latter problem, the operator $B:=-\divv(A(x)\nabla)$ has a nonnegative self-adjoint realization with form domain $H^1$.
By spectral calculus, interpolation and duality, we obtain
\begin{align*}
\sup_{t\in[0,T]}\big(\|u(t,\cdot)\|_{H^s}
+\|\partial_tu(t,\cdot)\|_{H^{s-1}}\big)
\leqslant C_T\big(\|u_0\|_{H^s}+\|u_1\|_{H^{s-1}}\big)
\end{align*}
for every $0\leqslant s\leqslant1$.
\end{enumerate}

\section*{Acknowledgments}
Wenhui Chen is supported in part by the National Natural Science Foundation of China (grant No. 12301270) and the Guangdong Basic and Applied Basic Research Foundation (grant No. 2025A1515010240). The authors thank Daniele Del Santo for drawing their attention to the early result in \cite{Yamazaki=1990}.

\end{document}